\pdfoutput=1
\documentclass[a4paper,11pt]{article}
\usepackage[margin=2.5cm]{geometry}

\usepackage[english]{babel}
\usepackage{amsfonts, amsthm, amsmath, amssymb, mathrsfs}
\usepackage{xcolor}

\usepackage{enumitem}
\newlist{steps}{enumerate}{1}
\setlist[steps, 1]{label = \textbf{Step \arabic*}:}
\usepackage{mathtools, tikz, tikz-cd, float}
\usepackage{subfig} 
\usepackage{ytableau} 
\usetikzlibrary { decorations.pathmorphing, decorations.pathreplacing}
\usepackage{hyperref}
\hypersetup{
colorlinks=true,
linkcolor=blue,
filecolor=magenta,      
urlcolor=cyan,
citecolor=[RGB]{0,204,0},
}
\usepackage[capitalise, nameinlink]{cleveref}
\usepackage{resizegather}
\usepackage[activate={true, nocompatibility}, final, tracking=true, kerning=true, spacing=true, factor=1100, stretch=10, shrink=10]{microtype}
\usepackage{authblk}
\usepackage{tikz}
\usepackage{dsfont}
\usetikzlibrary{matrix}

\usepackage{enumitem}
\setlist[enumerate,1]{label={(\roman{enumi})}, leftmargin=*}

\theoremstyle{definition}
\newtheorem{theorem}{Theorem}[section]
\newtheorem{definition}[theorem]{Definition}
\newtheorem{prin}[theorem]{Principle}

\newtheorem{example}[theorem]{Example}
\newtheorem{remark}[theorem]{Remark}
\newtheorem{prop}[theorem]{Proposition}
\newtheorem{corollary}[theorem]{Corollary}
\newtheorem{lemma}[theorem]{Lemma}

\usepackage[T1]{fontenc}
\usepackage[utf8]{inputenc}
\usepackage{babel}

\usepackage{subfiles}
\usepackage{authblk}

\tikzset{
  symbol/.style={
    draw=none,
    every to/.append style={
      edge node={node [sloped, allow upside down, auto=false]{$#1$}}}
  }
}

\microtypecontext{spacing=nonfrench}

\begin{document}
\title{\textbf{Dimension bounds for relative character varieties on the projective line with three punctures $G=GL(r), O(r), Sp(r)$}}
\author{Emmett Lennen}
\affil{Department of Mathematics, University of Pennsylvania\\ \url{elennen@sas.upenn.edu}}
\date{}
\maketitle
\abstract{We consider relative character varieties on $\mathbb{P}^1\backslash\{0,1,\infty\}$ with $G=GL(r),O(r),$ or $Sp(r)$. Using a diagrammatic method of Simpson's \cite{simp_quadpaper}, we give an explicit linear upper bound $R(d)$ on the rank r of an MC-minimal character variety of dimension $d>2$. An arbitrary character variety is isomorphic, via Katz's middle convolution, to one satisfying the bound. For the general linear and non-overlapping quadratic cases, the bounds we give are the sharpest possible using this method.}
\tableofcontents

\section{Introduction}

Let $O(r)$ (resp. $Sp(r)$) denote the group of linear transformations preserving a symmetric (resp. antisymmetric) nondegenerate bilinear form on a complex $r$-dimensional vector space. Consider an irreducible $G$-local system on $\mathbb{P}^1\backslash \{t_1,...,t_k\}$ with $G=GL(r)=GL(r,\mathbb{C})$, $O(r)$, or $Sp(r)$. The monodromy representation of the local system is an irreducible representation of the fundamental group
$$\rho:\pi_1(\mathbb{P}^1\backslash\{t_1,...,t_k\},t_0)\rightarrow G$$
where $t_0$ is a basepoint. Let $C_i\subset G$ denote the conjugacy class of the local monodromy transformation around $t_i$. Fix conjugacy classes $C_1,...,C_k$. The relative character variety is the variety of equivalence classes of such representations, i.e., the quotient of the variety of such representations by $G$, acting by simultaneous conjugation on the monodromies. When $G=O(r)$ or $Sp(r)$, we refer to these as quadratic relative character varieties.

In the general linear case, Katz \cite{katz_rigid} defined an operation called middle convolution with respect to a convolutor which gives isomorphisms between relative character varieties on $\mathbb{P}^1\backslash\{t_1,...,t_k\}$ with different monodromies over $\mathbb{Q}_l$. This was given an algebraic description over $\mathbb{C}$ by Dettweiler-Reiter in \cite{dett_reiter_alg_Katz_inverGalois}. Katz used middle convolution to give an algorithm to check the existence of rigid local systems. This is recovered by Dettweiler-Reiter's construction. A rigid local system can be described as one where the relative character variety containing it has dimension zero. The key part of the algorithm is that any dimension zero character variety of any rank is isomorphic to a character variety of rank $1$ via middle convolution. We refer to a character variety such that middle convolution cannot lower its rank as an MC-minimal character variety (see Definition \ref{def_mcminimalgl} for the exact definition when $G=GL(r)$). It is then an interesting question to ask how the rank of the MC-minimal character variety compares to its dimension. 
 
This paper focuses on the case of having three punctures, $k=3$. Let $\lambda_i^0\geq \lambda_i^1\geq...\geq \lambda_i^{s_i}$ be the partition of $r$ given by the multiplicities of the eigenvalues of the monodromy at $t_i$. In \cite{simpson_katzmiddleconv} Section 2.7, Simpson shows that for the general linear case, the dimension of an MC-minimal character variety can be read diagrammatically as the area leftover after removing squares of size $\lambda_i^j$ for $i=1,2,3$ and $j=1,...,s_i$ from an $r\times r$ square (see Section \ref{subsec_GLsetup}). The $r\times r$ square is split into three distinct columns where the $i$th column contains all the squares of size $\lambda_i^j$ for $j=1,...,s_i$. 

\begin{figure}[h]
\begin{center}
\begin{tikzpicture}[scale=0.7]
\filldraw[color=white, thin] (2.5,0) rectangle (6,6);
\draw[very thick] (0,0) rectangle (6,6);
\draw[dashed] (2.5,0) -- (2.5,6);
\filldraw[color=black, fill=gray!15, thick] (0,6) rectangle (2.5,3.5) node[pos=.5] {$\lambda^0_1$}; 
\filldraw[color=black, fill=gray!15, thick] (0,3.5) rectangle (1.5,2) node[pos=.5] {$\lambda^1_1$};
\filldraw[color=black, fill=gray!15, thick] (0,1) rectangle (1,2) node[pos=.5] {$\lambda^1_2$};
\filldraw[color=black, fill=gray!15, thick] (0,1) rectangle (1,0) node[pos=0.5] {$\lambda^{3}_1$};
\filldraw[color=black, fill=gray!15, thick] (4.5,4) rectangle (2.5,6) node[pos=.5] {$\lambda^0_2$}; 
\filldraw[color=black, fill=gray!15, thick] (2.5,4) rectangle (4,2.5) node[pos=.5] {$\lambda^1_2$}; 
\filldraw[color=black, fill=gray!15, thick] (2.5,2.5) rectangle (4,1) node[pos=.5] {$\lambda^2_2$}; 
\filldraw[color=black, fill=gray!15, thick] (2.5,0) rectangle (3.5,1) node[pos=.5] {$\lambda^{3}_2$}; 
\filldraw[color=black, fill=gray!15, thick] (4.5,6) rectangle (6,4.5) node[pos=.5] {$\lambda^0_3$}; 
\filldraw[color=black, fill=gray!15, thick] (4.5,3) rectangle (6,4.5) node[pos=.5] {$\lambda^1_3$}; 
\filldraw[color=black, fill=gray!15, thick] (4.5,3) rectangle (5.5,2) node[pos=.5] {$\lambda^2_3$}; 
\filldraw[color=black, fill=gray!15, thick] (4.5,2) rectangle (5.5,1) node[pos=.5] {$\lambda^3_3$}; 
\filldraw[color=black, fill=gray!15, thick] (4.5,0) rectangle (5.5,1) node[pos=.5] {$\lambda^{4}_3$}; 
\end{tikzpicture}
\end{center}
\caption{Example of a diagram associated to a $GL(r)$ relative character variety.}
\end{figure}
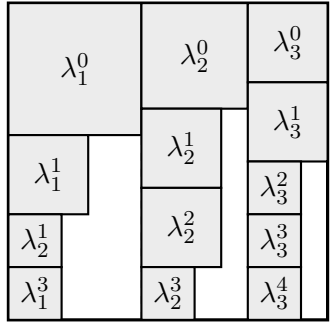

Using these configurations, we prove in Theorem \ref{thm_GLrbound_charvarversion} that any MC-minimal character variety with $G=GL(r)$ of dimension $d$ has rank $r\leq 3d$. 

In \cite{simp_quadpaper}, Simpson proved the existence of such a bound for the quadratic case with $k=3$ by giving a similar combinatorial description. In order to preserve quadratic local systems, one restricts to middle convolution with quadratic convolutors. However, it is no longer sufficient to consider middle convolution with a rank $1$ convolutor. Instead, Simpson showed that certain compositions of two rank $1$ (not necessarily quadratic) middle convolutions can be interpreted as a quadratic rank $2$ middle convolution. Using these higher rank middle convolutions, any MC-minimal quadratic character variety satisfies certain inequalities, which allows one to again read off the dimension diagrammatically as the leftover area after removing squares of size the multiplicity of the eigenvalues, but with extra linear terms (see Section \ref{subsec_quadsetup}). 

By tracing through Simpson's proof, the bound he finds is $\frac{3}{2}(21252d+11173142784)$ or, if one more carefully applies Prop 6.2 of \cite{simp_quadpaper}, the bound can be reduced to $\frac{3}{2}(1260d+135266208)$. We improve this bound by proving that any MC-minimal quadratic character variety of dimension $d$ has rank $r\leq 9d+54$ (Theorem \ref{thm_quadthm_charvar}). We accomplish this by proving a weaker version of Simpson's reductions and checking more cases to compensate. Just as Simpson does, we prove this by first finding a bound for the so-called non-overlapping configurations, then passing to the overlapping ones as well by multiplying the bound by $\frac{3}{2}$ (see Section \ref{subsec_quadsetup}). Let $q_i$ be the width of the column associated to the point $t_i$, which is often the biggest multiplicity of an eigenvalue of $C_i$. The key computational proposition in \cite{simp_quadpaper} can be stated as follows. Write $r=B_iq_i+\alpha_i$ where $B_i$ is the number of big parts (parts such that $\lambda_i^j\geq q_i/2$). For a column partition with dimension sufficiently small and $I\subset \{1,2,3\}$ with $q_i$ sufficiently big then $\sum_{i\in I}\frac{|\alpha_i|}{B_i}\leq 1$ which restricts the choices of $B_i,\alpha_i$. The remaining cases are then dealt with separately. This heavily reduces the possible partitions since the inequality restricts how far off parts can be from being size $0$ or $q$.

In this paper, we replace the notion of $B_i,\alpha_i$ by writing $r=n_iq_i+c_i$ where $c_i\in (-\frac{q_i}{2},\frac{q_i}{2}]$. One should think of $c_i$ as a measure of how close $q_i$ is to dividing $r$. We prove a version of Simpson's key proposition in Prop \ref{prop_nnpair}, restricting the possible column widths to ones with $(n_1,n_2)\in \{(2,2),(3,3),(2,4),(2,3)\}$ and $|c_i|\leq n_i$ with $i=1,2$ without reference to a particular partition. To reduce the possible partitions, we prove useful computational results such as proving that a partition higher in the dominance ordering has smaller dimension (Prop \ref{prop_quadDelta_domorder}) and finding all possible dimension contributions of a given column partition (Lemma \ref{lem_colqq_qqa}). This allows us to check the remaining cases in Section \ref{sec_5}. With this careful approach, we find the strictest bound for non-overlapping configurations (Theorem \ref{thm_quadthm_configsec2}).

\subsubsection*{Issue of non-emptiness:} To any character variety, we can associate to it a configuration. In the proofs of the theorems, we only work with configurations. However, in order for a character variety to achieve a certain dimension suggested by a configuration, one must show that it is nonempty. For $G=GL(r)$, the issue of showing non-emptiness of such character varieties is known as the Deligne-Simpson problem in the following form. Given $C_1,...,C_k$ conjugacy classes in $GL(r)$, find the necessary and sufficient conditions for there to exist $A_i\in C_i$ with no common proper invariant subspace such that
$$\prod_{i=1}^kA_i=\mathrm{Id}.$$
Simpson presented this problem in \cite{simp_product_of_matrices} and solved it when one of the conjugacy classes $C_k$ is semisimple with generic distinct eigenvalues. A sufficient and conjecturally necessary criterion was given by relating solutions of the Deligne-Simpson problem to representations of deformed multiplicative preprojective algebras in \cite{cb_shaw_multpreproj_middle_conv_DS}. This conjecture was recently proved in \cite{shu_tame_DS_problem} and \cite{cb_hubery_ds_prob}. See \cite{kostov_delignesimpson_survey} for a survey. One should be able to use the criteria to check non-emptiness in the general linear case, although the author was unable to do so. To the author's knowledge, little is known about the analogous problem for conjugacy classes of $G_{\varepsilon}=O(r),Sp(r)$.

\subsubsection*{Description of Sections:}
In Section \ref{sec_2}, background on MC-minimal character varieties and the combinatorial setup is given. The general linear case is described in Subsection \ref{subsec_GLsetup}, the quadratic case is described in Subsection \ref{subsec_quadsetup}, and a description of middle convolution is given in Subsection \ref{subsec_middleconv}. Section \ref{sec_3} proves the general linear bound. Section \ref{sec_4} sets up the needed computational and reduction results for the quadratic case. Section \ref{sec_5} runs through the leftover quadratic cases of $(n_1,n_2)\in \{(2,2),(3,3)(2,4),(2,3)\}$ with $|c_i|\leq n_i$ for $i=1,2$ and is where we find our minimal non-overlapping cases.

\subsection*{Acknowledgements} I would like to thank my advisor Ron Donagi for helpful discussions and Xinxuan Wang for suggesting to look at the dominance ordering and for giving the proof for Prop \ref{prop_mincoldim}. I would like to especially thank Carlos Simpson for patiently answering questions regarding the proofs in his paper and for providing helpful comments. 

\section{Minimal Character Varieties}\label{sec_2}

In Subsection \ref{subsec_GLsetup} and Subsection \ref{subsec_quadsetup}, we describe the combinatorial setup for our MC-minimal character varieties in the general linear and quadratic settings without explicitly describing middle convolution. We describe middle convolution in Subsection \ref{subsec_middleconv}.

\subsection{General Linear MC-Minimal Character Varieties:}\label{subsec_GLsetup}

We explain how to get configurations in the general linear setting. We follow Section 2 of \cite{simpson_katzmiddleconv}. The relative character variety on $\mathbb{P}^1\backslash \{t_1,...,t_k\}$ with monodromy $C_1,...,C_k$ of rank $r$ is given by
$$\mathscr{X}^{\mathrm{irred}}(r,C_1,...,C_k)=\{\rho:\pi_1(\mathbb{P}^1\backslash \{t_1,...,t_k\})\rightarrow GL(r) \text{ such that }\rho\text{ is irred, } \rho(\gamma_i)\in C_i\}/GL(r).$$
Here $C_i$ is a conjugacy class in $GL(r)$ and $\gamma_i$ is the monodromy transformation around the point $t_i$. Equivalently, one can think of this as the moduli of irreducible local systems on $\mathbb{P}^1\backslash\{t_1,...,t_k\}$ with monodromy around $t_i$ in $C_i$. 

\begin{prop}\label{prop_dimX}(Proposition 2.11 in \cite{simpson_katzmiddleconv})
    Assuming the character variety is non-empty. Then 
    $$\mathrm{dim}\mathscr{X}^{\mathrm{irred}}(r,C_1,...,C_k)=\sum_{i=1}^k\mathrm{dim}C_i-2r^2+2.$$
    \qed
\end{prop}

What is the dimension of a conjugacy class? We can write the dimension in terms of $Z(C_i)$, the centralizer of a matrix in $C_i$, as 
$$\mathrm{dim}C_i=r^2-\mathrm{dim}Z(C_i).$$
Assume all of the $C_i$ are semisimple. Let $\lambda_i^0\geq ...\geq \lambda_i^{s_i}$ be the multiplicities of the eigenvalues of matrices in $C_i$. Then the centralizer can be described as block diagonal matrices of the size given by $\lambda_i^j$. 

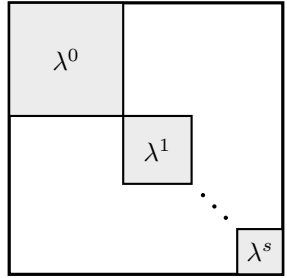
\begin{center}
\begin{tikzpicture}[scale=0.6]
\draw[very thick] (0,0) rectangle (6,6);
\filldraw[color=black, fill=gray!15, thick] (0,6) rectangle (2.5,3.5) node[pos=.5] {$\lambda^0$};
\filldraw[color=black, fill=gray!15, thick] (2.5,3.5) rectangle (4,2) node[pos=.5] {$\lambda^1$};
\filldraw[black] (4.5,1.5) circle (1pt);
\filldraw[black] (4.25,1.75) circle (1pt);
\filldraw[black] (4.75,1.25) circle (1pt);
\filldraw[color=black, fill=gray!15, thick] (5,1) rectangle (6,0) node[pos=0.5] {$\lambda^s$};
\end{tikzpicture}
\end{center}

Hence, the dimension of $C_i$ is the number of empty boxes in the above picture. 

\begin{remark}
    When the conjugacy classes are not semisimple, the centralizer is smaller, and hence the dimension of such character varieties is bigger. The proof of our bound only relies on the sizes of the boxes, and so the bound will also work for the non-semisimple case.
\end{remark}

Move all the boxes to the left.

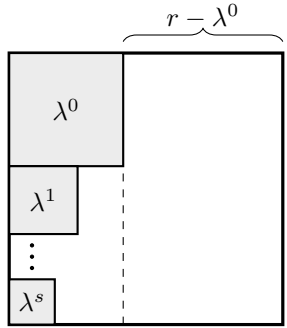
\begin{center}
\begin{tikzpicture}[scale=0.6]
\draw[decorate,decoration={brace,amplitude=5pt,mirror,raise=1ex}] (6,6) -- (2.5,6) node[above=6pt, midway]{$r-\lambda^0$};
\filldraw[color=white, thin] (2.5,0) rectangle (6,6);
\draw[very thick] (0,0) rectangle (6,6);
\draw[dashed] (2.5,0) -- (2.5,6);
\filldraw[color=black, fill=gray!15, thick] (0,6) rectangle (2.5,3.5) node[pos=.5] {$\lambda^0$}; 
\filldraw[color=black, fill=gray!15, thick] (0,3.5) rectangle (1.5,2) node[pos=.5] {$\lambda^1$};
\filldraw[black] (0.5,1.5) circle (1pt);
\filldraw[black] (0.5,1.75) circle (1pt);
\filldraw[black] (0.5,1.25) circle (1pt);
\filldraw[color=black, fill=gray!15, thick] (0,1) rectangle (1,0) node[pos=0.5] {$\lambda^s$};
\end{tikzpicture}
\end{center}

Then $\mathrm{dim}C_i\geq r(r-\lambda_i^0)$ by considering the empty right rectangle as in the picture and by comparing this to Prop \ref{prop_dimX},
$$\mathrm{dim}\mathscr{X}^{\mathrm{irred}}(r,C_1,...,C_k)\geq \sum_{i=1}^k r(r-\lambda_i^0)-2r^2+2=2+r\left((k-2)r-\sum_{i=1}^k\lambda_i^0\right).$$
Then the defect is defined as
\begin{equation}\label{eq_delta}
\delta(C_1,...,C_k)\coloneqq (k-2)r-\sum_{i=1}^k\lambda_i^0.
\end{equation}
Katz's middle convolution operations, defined in Katz's book \cite{katz_rigid}, give isomorphisms between different relative character varieties on $\mathbb{P}^1\backslash\{t_1,...,t_k\}$. In Subsection \ref{subsec_middleconv}, we will give the explicit construction of middle convolution. Under convolution associated to the highest multiplicity eigenvalue, the rank changes as
$$r\mapsto r+\delta(C_1,...,C_k).$$
Hence, we keep applying until $\delta\geq 0$. Note that this process will terminate since the rank will decrease. We refer to a character variety as an MC-minimal character variety when applying middle convolution will no longer decrease the rank, i.e., the defect is non-negative.

\begin{definition}\label{def_mcminimalgl}
    A relative character variety with $\delta\geq 0$ is called MC-minimal.
\end{definition}

For $k=3$, this implies the following, which we refer to as the ``fitting into a square" property.

\begin{lemma}\label{lem_glrfitsquareproperty}(``fitting into a square" property)
    MC-minimal character varieties with $k=3$ satisfy $\lambda_1^0+\lambda_2^0+\lambda_3^0\leq r$.
\end{lemma}
\begin{proof}
    This is the $k=3$ case of the formula for the defect, Equation \ref{eq_delta}.
\end{proof}

For the rest of the section, we specialize to the case of $k=3$. The dimension of our character variety is
\begin{equation}\label{eq_dim_charvar}
\sum_{i=1}^3\mathrm{dim}C_i-2r^2+2=3r^2-\sum_{i=1}^3\sum_{j=0}^{s_i}(\lambda_i^j)^2-2r^2+2=r^2-\sum_{i=1}^3\sum_{j=0}^{s_i}(\lambda_i^j)^2+2=\Delta+2
\end{equation}
where $\Delta\coloneqq r^2-\sum_{i=1}^3\sum_{j=0}^{s_i}(\lambda_i^j)^2$. Since we have the ``fitting into a square" property, this means that any MC-minimal character variety is associated to the data
\begin{itemize}
    \item $r=q_1+q_2+q_3$ (split the $r\times r$ square into three columns of width $q_i$)
    \item $r=\sum_{j=0}^{s_i}\lambda_i^j$ with $\lambda_i^j\leq q_i$ for $i=1,2,3$
    \item $\lambda_i^0\geq \lambda_i^1\geq ...\geq \lambda_i^{s_i}$ for $i=1,2,3$
    \item the (box) dimension is the number of empty boxes, $\Delta=r^2-\sum_{i=1}^3\sum_{j=0}^{s_i}(\lambda_i^j)^2$.
\end{itemize}

\begin{example}\label{ex_gl_dim}
    Here is an example of a configuration with $r=10$, $q_1=5, q_2=3$, and $q_3=2$ given in Figure \ref{pic_example_gldim}. The first column is $(5,5)$, the second is $(3,3,3,1)$, and the third is $(2,2,2,2,1,1)$. The box dimension is $4$ with no empty boxes in the first column and two empty boxes in each of the second and third columns.
    
\begin{figure}[h]
\begin{center}
\begin{tikzpicture}[scale=0.7]
\draw[very thick] (0,0) rectangle (7,7);
\filldraw[color=black, fill=gray!15, thick] (0,7) rectangle (3.5,3.5) node[pos=.5] {$5$};
\filldraw[color=black, fill=gray!15, thick] (0,0) rectangle (3.5,3.5) node[pos=.5] {$5$};
\filldraw[color=black, fill=gray!15, thick] (3.5,7) rectangle (5.6,4.9) node[pos=.5] {$3$}; 
\filldraw[color=black, fill=gray!15, thick] (5.6,4.9) rectangle (3.5,2.8) node[pos=.5] {$3$}; 
\filldraw[color=black, fill=gray!15, thick] (5.6,0.7) rectangle (3.5,2.8) node[pos=.5] {$3$}; 
\filldraw[color=black, fill=gray!15, thick] (3.5,0.7) rectangle (4.2,0) node[pos=.5] {$1$}; 
\filldraw[color=black, fill=gray!15, thick] (5.6,7) rectangle (7,5.6) node[pos=.5] {$2$}; 
\filldraw[color=black, fill=gray!15, thick] (5.6,4.2) rectangle (7,5.6) node[pos=.5] {$2$}; 
\filldraw[color=black, fill=gray!15, thick] (5.6,4.2) rectangle (7,2.8) node[pos=.5] {$2$}; 
\filldraw[color=black, fill=gray!15, thick] (5.6,1.4) rectangle (7,2.8) node[pos=.5] {$2$}; 
\filldraw[color=black, fill=gray!15, thick] (5.6,1.4) rectangle (6.3,0.7) node[pos=.5] {$1$}; 
\filldraw[color=black, fill=gray!15, thick] (5.6,0) rectangle (6.3,0.7) node[pos=.5] {$1$}; 
\end{tikzpicture}
\caption{Diagram for Example \ref{ex_gl_dim}.}
\label{pic_example_gldim}
\end{center}
\end{figure}
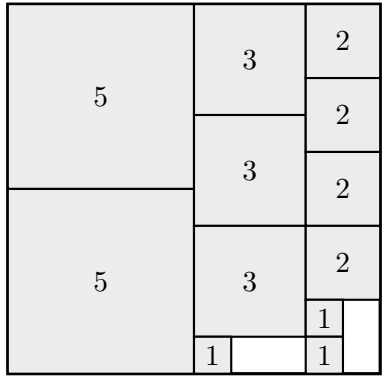
\end{example}

\begin{definition}\label{def_glrconfigurationdata}
    Data such as above will be called $GL(r)$-configuration data for rank $r$ local systems on $\mathbb{P}^1\backslash \{t_1,t_2,t_3\}$.
\end{definition}

Note that in \cite{simp_quadpaper}, this would be called minimal partition data, but we will be talking about partitions that minimize dimension within a column. To avoid confusion, we just say configuration instead. Using this configuration description, we get a lower bound on the rank based on the dimension.

\begin{theorem}\label{thm_GLrboundsec2}(Theorem \ref{thm_GLrbound})
    Given $\Delta>0$ even, any $GL(r)$-configuration of dimension $\Delta$ has rank $r\leq 3\Delta+6$ and $r=3\Delta+6$ is achieved for exactly the configuration with $(q_1,q_2,q_3)=\left(\frac{r}{2}, \frac{r}{3}, \frac{r}{6}\right)$ and columns $\left(\frac{r}{2},\frac{r}{2}\right),\left(\frac{r}{3},\frac{r}{3},\frac{r}{3}\right),$ and $\left(\frac{r}{6},\frac{r}{6},\frac{r}{6},\frac{r}{6},\frac{r}{6},\frac{r}{6}-1,1\right)$.
    \begin{figure}[h]\label{pic_minglcases}
    \begin{center}
        \begin{tikzpicture}[scale=1]
        \draw[very thick] (0,0) rectangle (7,7);
        \filldraw[color=black, fill=gray!15, thick] (0,7) rectangle (3.5,3.5) node[pos=.5] {$\frac{r}{2}$};
        \filldraw[color=black, fill=gray!15, thick] (0,0) rectangle (3.5,3.5) node[pos=.5] {$\frac{r}{2}$};
        \filldraw[color=black, fill=gray!15, thick] (3.5,7) rectangle (5.83,4.66) node[pos=.5] {$\frac{r}{3}$};
        \filldraw[color=black, fill=gray!15, thick] (3.5,2.33) rectangle (5.83,4.66) node[pos=.5] {$\frac{r}{3}$};
        \filldraw[color=black, fill=gray!15, thick] (3.5,2.33) rectangle (5.83,0) node[pos=.5] {$\frac{r}{3}$};
        \filldraw[color=black, fill=gray!15, thick] (5.83,7) rectangle (7,5.83) node[pos=.5] {$\frac{r}{6}$};
        \filldraw[color=black, fill=gray!15, thick] (5.83,4.66) rectangle (7,5.83) node[pos=.5] {$\frac{r}{6}$};
        \filldraw[color=black, fill=gray!15, thick] (5.83,4.66) rectangle (7,3.5) node[pos=.5] {$\frac{r}{6}$};
        \filldraw[color=black, fill=gray!15, thick] (5.83,2.33) rectangle (7,3.5) node[pos=.5] {$\frac{r}{6}$};
        \filldraw[color=black, fill=gray!15, thick] (5.83,2.33) rectangle (7,1.16) node[pos=.5] {$\frac{r}{6}$};
        \filldraw[color=black, fill=gray!15, thick] (5.83,1.16) rectangle (6.61,0.388) node[pos=.5] {\footnotesize{$\frac{r}{6}-1$}};
        \filldraw[color=black, fill=gray!15, thick] (5.83,0.388) rectangle (6.22,0) node[pos=.5] {$1$};
        \end{tikzpicture}
        \caption{Diagram for the $GL(r)$ configurations that achieve minimal dimension.}
    \end{center}
    \end{figure}
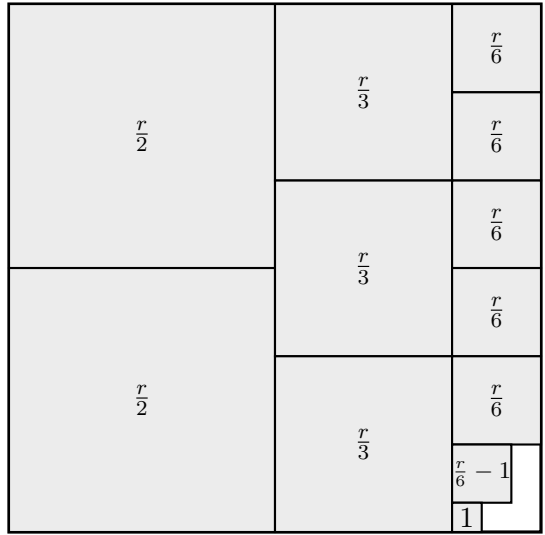
\end{theorem}

From the above discussion and Equation \ref{eq_dim_charvar}, it follows that
\begin{theorem}\label{thm_GLrbound_charvarversion}
    For $d>2$ even, any MC-minimal character variety with $G=GL(r)$ of dimension $d$ on $\mathbb{P}^1\backslash \{t_1,t_2,t_3\}$ has rank $r\leq 3d$. 
\end{theorem}

While we will show that the bound is strict on the level of configurations, to show strictness for character varieties, one needs to deal with the problem of non-emptiness, as discussed in the introduction. If there is a character variety with the configuration giving equality, then the bound would be strict for the character variety level as well.

\subsection{Quadratic MC-minimal Character Varieties}\label{subsec_quadsetup}

Recall that we denote by $O(r)$ the group of linear transformations preserving a symmetric nondegenerate bilinear form on an $r$-dimensional complex vector space and $Sp(r)$ the group preserving an anti-symmetric nondegenerate bilinear form instead. We summarize the set-up and results from \cite{simp_quadpaper}. A quadratic relative character variety on $\mathbb{P}^1\backslash\{t_1,...,t_k\}$ with monodromy $C_1,...,C_k$ is given by
$$\mathscr{X}^{\mathrm{irred}}(r,\varepsilon,C_1,...,C_k)=\{\rho:\pi_1(\mathbb{P}^1\backslash\{t_1,...,t_k\})\rightarrow G_{\varepsilon}\text{ such that }\rho\text{ is irred, }\rho(\gamma_i)\in C_i\}/G_{\varepsilon}.$$
Here $\varepsilon=\pm 1$ where $G_{1}=O(r)$ and $G_{-1}=Sp(r)$. We have denoted by $C_i$ conjugacy classes in $G_{\varepsilon}$ and $\gamma_i$ the monodromy transformation around the point $t_i$. Equivalently, one can think of this as the moduli of irreducible quadratic local systems of type $\varepsilon$ on $\mathbb{P}^1\backslash\{t_1,...,t_k\}$ with monodromy around $t_i$ in $C_i$ where a quadratic local system of type $\varepsilon$ is a local system that is isomorphic to its dual and has monodromy in $G_{\varepsilon}$.

\begin{lemma}(Lemma 2.4 of \cite{simp_quadpaper})
    The dimension of $\mathscr{X}^{\mathrm{irred}}(r,\varepsilon, C_1,...,C_k)$ is
    $$\mathrm{dim}\mathscr{X}^{\mathrm{irred}}(r,\varepsilon,C_1,...,C_k)=\mathrm{dim}(C_1)+...+\mathrm{dim}(C_k)-2\mathrm{dim}(G_{\varepsilon}).$$\qed
\end{lemma}

Unlike in the general linear situation, we have three distinct types of eigenvalues that affect the dimension of the conjugacy classes differently. 
\begin{itemize}
    \item The first type, which we refer to as type $m$, are paired eigenvalues $z\neq z^{-1}$. The eigenvalues $z,z^{-1}$ will occur with the same multiplicity. We denote by $m_i^j$ the multiplicity of the $j$th pair of type $m$ eigenvalues at $t_i$.
    \item The second and third types, which we refer to as type $e$ and $f$ respectively, are half-unital eigenvalues where $z=z^{-1}$. This means that we are looking at Jordan blocks of the eigenvalues $1$ or $-1$. At point $i$, we take the eigenvalue with the biggest Jordan block to be type $e$. We refer to their multiplicities as $e_i^j$ and $f_i^j$ respectively for the size of the $j$th Jordan block of that type at point $t_i$.
\end{itemize}

Unless otherwise stated, we will always order the multiplicities of the eigenvalues of each type in decreasing order, i.e., $e_i^0\geq e_i^1\geq...\geq e_i^{s_{i,e}}$, $f_i^0\geq f_i^1\geq ...\geq f_i^{s_{i,f}}$, and $m_0^0\geq ...\geq m_i^{s_{i,m}}$. Simpson computes the dimension using these multiplicities.
\begin{prop}(Prop 2.5 of \cite{simp_quadpaper})
    Suppose $\mathcal{F}$ is an irreducible quadratic local system of rank $r$ of type $\varepsilon$. Twice the dimension at $\mathcal{F}$ of the character variety $\mathscr{X}^{\mathrm{irred}}(r,\varepsilon,C_1,...,C_k)$ is
    $$(k-2)r^2-2\sum_{i,j}(m_i^j)^2-\sum_{i,j}(e_i^j)^2-\sum_{i,j}(f_i^j)^2-\varepsilon\left((k-2)r-\sum_{i,j}(-1)^{j}e_i^j-\sum_{i,j}(-1)^{j}f_i^j\right).$$\qed
\end{prop}

For the rest of this section, we assume $k=3$. Then the above says
\begin{equation}\label{eq_Deltaquadratic}
2\mathrm{dim}\mathscr{X}^{\mathrm{irred}}(r,\varepsilon,C_1,C_2,C_3)=r^2-2\sum_{i,j}(m_i^j)^2-\sum_{i,j}(e_i^j)^2-\sum_{i,j}(f_i^j)^2-\varepsilon\left(r-\sum_{i,j}(-1)^je_i^{j}-\sum_{i,j}(-1)^{j}f_i^j\right).
\end{equation}
We will define $\Delta=\beta-\varepsilon l$ as the right hand side where $\beta=r^2-2\sum_{i,j}(m_i^j)^2-\sum_{i,j}(e_i^j)^2-\sum_{i,j}(f_i^j)^2$ and $l=r-\sum_{i,j}(-1)^je_i^{j}-\sum_{i,j}(-1)^{j}f_i^j$. Later on, we will refer to $\beta$ as the box dimension and $l$ as the linear correction. Note that for the general linear case, $\Delta=\beta$ with no linear correction.

\begin{definition}\label{def_mcminimalquad}
    An irreducible quadratic character variety is called MC-minimal if its rank does not decrease under middle convolution using quadratic MC-invertible convolution kernels.
\end{definition}

Quadratic MC-minimal character varieties do not have a convenient description like the general linear case. However, in \cite{simp_quadpaper} Lemmas 4.1-4.5, Simpson shows that an irreducible MC minimal variety of rank $>2$ with $k=3$ satisfies certain similar inequalities to Lemma \ref{lem_glrfitsquareproperty}. He then calls these character varieties numerically MC-minimal.

\begin{definition}\label{def_numer_Mcmin_charvar}(\cite{simp_quadpaper})
    An irreducible character variety of rank $r>2$ is called numerically MC-minimal if it satisfies (up to permuting $i=1,2,3$)
    \begin{itemize}
        \item $m_1^j+m_2^j+m_3^j\leq r$ for all $j$
        \item $e_1^0+e_2^0+e_3^0\leq r$
        \item $m_1^j+m_2^j+e_3^0\leq r$ for all $j$
        \item $m_1^0+e_2^0+(e_3^0+f_3^0)/2\leq r$
        \item $m_1^0+e_2^0+(e_3^0+e_3^1)/2\leq r$.
    \end{itemize}
\end{definition}

He calls a choice of data $e_i^{0}\geq e_i^1\geq...\geq e_i^{s_{i,e}}, f_i^0\geq ...\geq f_i^{s_{i,f}}$, and $m_i^{0}\geq ...\geq m_i^{s_{i,m}}$ with $\sum_{j=0}^{s_{i,e}}e_i^j+\sum_{j=1}^{s_{i,f}}f_i^j+2\sum_{j=0}^{s_{i,m}}m_i^j=r$ for $i=1,2,3$ that satisfies the above inequalities, a numerically MC-minimal partition. We will refer to them as either numerically MC-minimal configurations or quadratic configurations. Simpson showed that this has a ``fitting into a square" property similar to the general linear case (Lemma \ref{lem_glrfitsquareproperty}), but now columns may overlap in a particular way. In the non-overlapping case, just like the general linear case, all squares fit into a rectangle $q_i\times r$ with $q_1+q_2+q_3=r$. In the overlapping case, there will be $q_i$ with $q_1+q_2+q_3\leq r$, where we think of $q_i$ as the width of the $i$th column, except possibly the biggest eigenvalue of the overlapping columns. In the following proposition, we drop our assumption that $q_1\geq q_2\geq q_3$. 

\begin{prop}\label{prop_quadfitsquare}(Prop 5.1 of \cite{simp_quadpaper}, fitting into a square property)
    Any numerically MC-minimal collection of partitions $(e_i^j, f_i^j, m_i^j)$ has the property that the union of squares of sizes $e_i^j\times e_i^j, f_i^j\times f_i^j$, and $m_i^j\times m_i^j$ fits into a square of size $r\times r$ in one of two ways.
    \begin{itemize}
        \item Non-overlapping case. Each collection of squares fits into a rectangle of size $q_i\times r$ with $q_1+q_2+q_3=r$.
        \item Overlapping case. There are two half-unital squares that overlap. Up to permutation of the points, assume these are $e_2^0$ and $e_3^0$, then the squares for $i=1$ fit into a rectangle of size $q_1\times r$ while the squares for $i=2,3$ are either the $e_i^0$ part or fit into a rectangle of size $q_i\times r_i$ with $r_i=r-e_i^0$. The $e_2^0$ and $e_3^0$ squares have a horizontal overlap of $\mu\geq 1$ so that $r=q_1+e_2^0+e_3^0-\mu$. Let $\nu=r-e_2^0-e_3^0$.
    \end{itemize}
\end{prop}

\begin{example}\label{ex_overlapping} 
    The following is an example of an overlapping numerically MC-minimal configuration. The diagram is given in Figure \ref{pic_overlappingexample}. The columns are given as $(7^m, 7^m, 4^e,2^e)$, $(10^e,6^e, 4^e)$, and $(6^e,3^m,3^m,3^f,3^f,2^e)$ where the exponent says what type the square of that size is. Here $r=20$, $q_1=7, q_2=6, q_3=3$, $\mu=3$, and $\nu=4$. The box dimension is $52$. The linear correction is $2$. 
    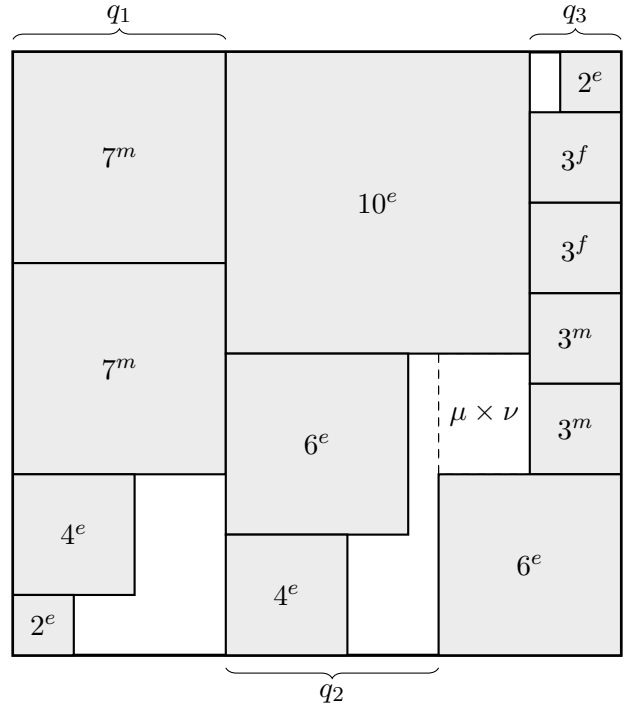
\begin{figure}[h]
    \centering
        \begin{tikzpicture}[scale=0.8]
        \draw[very thick] (0,0) rectangle (10,10);
        \filldraw[color=black, dashed, fill=white] (7,5) rectangle (8.5,3) node[pos=.5]{$\mu\times \nu$};
        \filldraw[color=black, fill=gray!15, thick] (0,10) rectangle (3.5,6.5) node[pos=.5] {$7^m$};
        \filldraw[color=black, fill=gray!15, thick] (0,3) rectangle (3.5,6.5) node[pos=.5] {$7^m$};
        \filldraw[color=black, fill=gray!15, thick] (0,3) rectangle (2,1) node[pos=.5] {$4^e$};
        \filldraw[color=black, fill=gray!15, thick] (0,0) rectangle (1,1) node[pos=.5] {$2^e$};
        \filldraw[color=black, fill=gray!15, thick] (3.5,10) rectangle (8.5,5) node[pos=.5] {$10^e$};
        \filldraw[color=black, fill=gray!15, thick] (6.5,5) rectangle (3.5,2) node[pos=.5] {$6^e$};
        \filldraw[color=black, fill=gray!15, thick] (5.5,0) rectangle (3.5,2) node[pos=.5] {$4^e$};
        \filldraw[color=black, fill=gray!15, thick] (10,10) rectangle (9,9) node[pos=.5] {$2^e$};
        \filldraw[color=black, fill=gray!15, thick] (10,9) rectangle (8.5,7.5) node[pos=.5] {$3^f$};
        \filldraw[color=black, fill=gray!15, thick] (10,7.5) rectangle (8.5,6) node[pos=.5] {$3^f$};
        \filldraw[color=black, fill=gray!15, thick] (10,6) rectangle (8.5,4.5) node[pos=.5] {$3^m$};
        \filldraw[color=black, fill=gray!15, thick] (10,3) rectangle (8.5,4.5) node[pos=.5] {$3^m$};
        \filldraw[color=black, fill=gray!15, thick] (7,3) rectangle (10,0) node[pos=.5] {$6^e$};
        \draw[dashed] (7,3) -- (7,5); 
        \draw[decorate,decoration={brace,amplitude=4pt, raise=1ex}] (7,0) -- (3.5,0) node[below=7pt, midway]{$q_2$}; 
        \draw[decorate,decoration={brace,amplitude=4pt, raise=1ex}] (0,10) -- (3.5,10) node[above=7pt, midway]{$q_1$}; 
        \draw[decorate,decoration={brace,amplitude=4pt, raise=1ex}] (8.5,10) -- (10,10) node[above=7pt, midway]{$q_3$}; 
        \end{tikzpicture}
        \caption{Diagram for an overlapping numerically MC-minimal quadratic configuration.}
        \label{pic_overlappingexample}
    \end{figure}
\end{example}

Simpson shows that for any overlapping configuration, there is a non-overlapping configuration of rank $\geq \frac{2r}{3}$ of the same dimension, which implies the following:
\begin{corollary}(Corollary 7.3 of \cite{simp_quadpaper})
    Let $R^{non}(d)$ be an upper bound for the rank of the non-overlapping case for any MC-minimal configuration of dimension $d$. Then any overlapping numerically MC-minimal configuration of dimension $d$ and rank $r$ has $r\leq \frac{3}{2}R^{non}(d)$.\qed
\end{corollary}

For the reader's convenience, we describe the associated non-overlapping configuration. It will be of rank $r-\mu$. Note that $\mu\leq \frac{r}{3}$ since $\mu\leq q_i$ for all $i$ so the Corollary will follow by the construction. Up to reordering the points (note that this is allowed to break our usual assumption that $q_1\geq q_2\geq q_3$), assume that the overlapping happens between the second and third columns. We have $m_1^0>e_1^0$ and $e_i^0>m_i^0$ for $i=2,3$. All squares in the second and third columns, except for $e_i^0$, fit inside a $q_i\times (r-e_i^0)$ rectangle with $q_i=e_i^0-\mu$. We can take $q_1=m_1^0$. Define the associated non-overlapping configuration as follows. 
\begin{itemize}
    \item 1st column: The first column has a pair of type $m_1^0$. Replace the first of this pair with a type $e$ part of the same size and the second of the pair with a type $e$ part of size $m_1^0-\mu$. Keep the rest of the parts in the first column the same. Note that $m_1^0-\mu=\nu>0$.
    \item 2nd and 3rd column: Replace the $e_i^0$ part with a type $e$ part of size $e_i^0-\mu$. Keep the rest of the parts in the $i$th column the same.
\end{itemize}
Now the first column still fits in $q_1\times r$ while the second and third columns fit inside $q_i\times r$ with $i=2,3$. The dimension will be the same since both the box dimension and linear correction are preserved (see Lemma 7.2 of \cite{simp_quadpaper}).

\begin{example}
   The associated non-overlapping configuration to Example \ref{ex_overlapping} is as follows. The diagram is given in Figure \ref{pic_overlappingexample_non}. The columns are given by $(7^e, 4^e, 4^e,2^e), (7^e, 6^e, 4^e)$, and $(3^e, 3^m, 3^m, 3^f, 3^f, 2^e)$. The size of the square is $r=20-\mu=17$. The box dimension and linear correction are the same as the overlapping one. Notice that the $\mu\times \nu$ box gets moved to the column that is not involved in the overlap.
   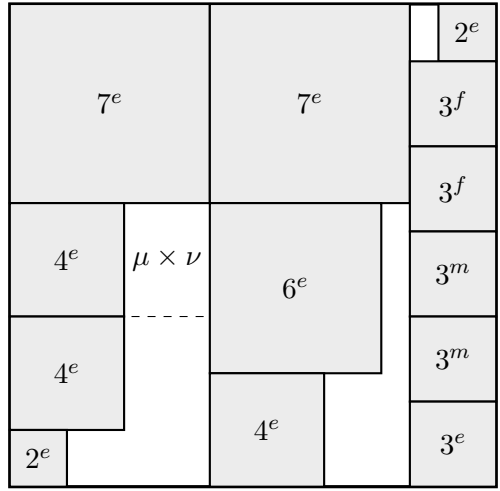
\begin{figure}[h]
    \centering
        \begin{tikzpicture}[scale=0.8]
        \draw[very thick] (0,0) rectangle (8,8);
        \filldraw[color=black, dashed, fill=white] (1.88,2.82) rectangle (3.29,4.70) node[pos=.5]{$\mu\times \nu$}; 
        \filldraw[color=black, fill=gray!15, thick] (0,8) rectangle (3.29,4.70) node[pos=.5] {$7^e$};
        \filldraw[color=black, fill=gray!15, thick] (0,4.70) rectangle (1.88,2.82) node[pos=.5] {$4^e$};
        \filldraw[color=black, fill=gray!15, thick] (0,0.94) rectangle (1.88,2.82) node[pos=.5] {$4^e$};
        \filldraw[color=black, fill=gray!15, thick] (0,0.94) rectangle (0.94,0) node[pos=.5] {$2^e$};
        \filldraw[color=black, fill=gray!15, thick] (3.29,8) rectangle (6.58,4.70) node[pos=.5] {$7^e$};
        \filldraw[color=black, fill=gray!15, thick] (3.29,4.70) rectangle (6.11,1.88) node[pos=.5] {$6^e$};
        \filldraw[color=black, fill=gray!15, thick] (3.29,0) rectangle (5.17,1.88) node[pos=.5] {$4^e$};
        \filldraw[color=black, fill=gray!15, thick] (8,8) rectangle (7.05,7.05) node[pos=.5] {$2^e$};
        \filldraw[color=black, fill=gray!15, thick] (8,7.05) rectangle (6.58,5.64) node[pos=.5] {$3^f$};
        \filldraw[color=black, fill=gray!15, thick] (8,5.64) rectangle (6.58,4.23) node[pos=.5] {$3^f$};
        \filldraw[color=black, fill=gray!15, thick] (8,4.23) rectangle (6.58,2.82) node[pos=.5] {$3^m$};
        \filldraw[color=black, fill=gray!15, thick] (8,1.41) rectangle (6.58,2.82) node[pos=.5] {$3^m$};
        \filldraw[color=black, fill=gray!15, thick] (8,1.41) rectangle (6.58,0) node[pos=.5] {$3^e$};
        \end{tikzpicture}
        \caption{Diagram for the non-overlapping associated to Example \ref{ex_overlapping}.}\label{pic_overlappingexample_non}
    \end{figure}
\end{example}

Hence, we focus on the non-overlapping case. Our combinatorial data associated to a numerically MC-minimal quadratic configuration can then be described as follows:
\begin{itemize}
    \item $r=q_1+q_2+q_3$
    \item $r=\sum_{j=0}^{s_i}\lambda_i^j$ with $\lambda_i^j\leq q_i$ for $i=1,2,3$
    \item Each $\lambda_i^j$ is assigned type $e$, type $f$, or type $m$. Then let $e_i^0\geq e_i^1\geq ...\geq e_i^{s_{i,e}}$, $f_i^0\geq f_i^1\geq...\geq f_i^{s_{i,f}}$, and $m_i^0\geq m_i^1\geq...\geq m_i^{s_{i,m}}$.
    \item The box dimension is
    $$\beta=r^2-\sum_{j=0}^{s_1}(\lambda_1^j)^2-\sum_{j=0}^{s_2}(\lambda_2^j)^2-\sum_{j=0}^{s_3}(\lambda_3^j)^2.$$
    \item The linear correction is 
    $$l=r-\sum_{i,j}(-1)^{j}e_i^j-\sum_{i,j}(-1)^{j}f_i^j.$$
    \item The dimension is
    $$\Delta=\beta-\varepsilon l$$
    with $\varepsilon=\pm 1$.
\end{itemize}

Notice that if our MC-minimal quadratic configuration comes from a non-empty MC-minimal character variety, then $\Delta$ is twice the dimension by Equation \ref{eq_Deltaquadratic}.

Different choices of linear correction produce different dimensions. In the following example, one can find choices of type $e$ and $f$ that flip the sign of the linear correction. The box dimension $\beta$ will be the same for both examples. This means that the same dimensions achieved for $\varepsilon=1$ are also achieved for $\varepsilon=-1$. In Lemma \ref{lem_colqq_qqa}, we find that this is the case for most configurations.

\begin{example}\label{ex_quad_dim}
    We use the same squares as in Example \ref{ex_gl_dim} of rank $r=10$, which has box dimension $4$, to give two examples of choices of linear correction. The configuration with columns $(5^e,5^f), (3^e,3^e,3^e,1^f),$ and $(2^e,2^f,2^e,2^e,1^e,1^e)$ has linear correction $l=-8$. The configuration with  columns $(5^e,5^f),(3^e,3^e,3^e,1^e),$ and $(2^e,2^e,2^e,2^e,1^e,1^e)$ has $l=8$. The diagrams are given in Figure \ref{pic_difflincorr}. Then dimensions are then $-4$ or $4$ respectively if $\varepsilon=1$ and $4$ or $-4$ if $\varepsilon=-1$. 
    \begin{figure}[h]%
    \centering
    \subfloat[\centering $l=r-10-4-4=-8$]{{\begin{tikzpicture}
        \draw[very thick] (0,0) rectangle (5,5);
        \filldraw[color=black, fill=gray!15, thick] (0,5) rectangle (2.5,2.5) node[pos=.5] {$5^e$};%
        \filldraw[color=black, fill=gray!15, thick] (0,0) rectangle (2.5,2.5) node[pos=.5] {$5^f$};
        \filldraw[color=black, fill=gray!15, thick] (2.5,5) rectangle (4,3.5) node[pos=.5] {$3^e$};
        \filldraw[color=black, fill=gray!15, thick] (2.5,2) rectangle (4,3.5) node[pos=.5] {$3^e$};
        \filldraw[color=black, fill=gray!15, thick] (2.5,2) rectangle (4,0.5) node[pos=.5] {$3^e$};
        \filldraw[color=black, fill=gray!15, thick] (2.5,0) rectangle (3,0.5) node[pos=.5] {$1^f$};
        \filldraw[color=black, fill=gray!15, thick] (4,5) rectangle (5,4) node[pos=.5] {$2^e$};
        \filldraw[color=black, fill=gray!15, thick] (4,3) rectangle (5,4) node[pos=.5] {$2^f$};
        \filldraw[color=black, fill=gray!15, thick] (4,3) rectangle (5,2) node[pos=.5] {$2^e$};
        \filldraw[color=black, fill=gray!15, thick] (4,1) rectangle (5,2) node[pos=.5] {$2^e$};
        \filldraw[color=black, fill=gray!15, thick] (4,1) rectangle (4.5,0.5) node[pos=.5] {$1^e$};
        \filldraw[color=black, fill=gray!15, thick] (4,0) rectangle (4.5,0.5) node[pos=.5] {$1^e$};
        \end{tikzpicture}
        }}%
    \qquad
    \subfloat[\centering $l=r-0-2-0=8$]{\begin{tikzpicture}
        \draw[very thick] (0,0) rectangle (5,5);
        \filldraw[color=black, fill=gray!15, thick] (0,5) rectangle (2.5,2.5) node[pos=.5] {$5^e$};%
        \filldraw[color=black, fill=gray!15, thick] (0,0) rectangle (2.5,2.5) node[pos=.5] {$5^e$};
        \filldraw[color=black, fill=gray!15, thick] (2.5,5) rectangle (4,3.5) node[pos=.5] {$3^e$};
        \filldraw[color=black, fill=gray!15, thick] (2.5,2) rectangle (4,3.5) node[pos=.5] {$3^e$};
        \filldraw[color=black, fill=gray!15, thick] (2.5,2) rectangle (4,0.5) node[pos=.5] {$3^e$};
        \filldraw[color=black, fill=gray!15, thick] (2.5,0) rectangle (3,0.5) node[pos=.5] {$1^e$};
        \filldraw[color=black, fill=gray!15, thick] (4,5) rectangle (5,4) node[pos=.5] {$2^e$};
        \filldraw[color=black, fill=gray!15, thick] (4,3) rectangle (5,4) node[pos=.5] {$2^e$};
        \filldraw[color=black, fill=gray!15, thick] (4,3) rectangle (5,2) node[pos=.5] {$2^e$};
        \filldraw[color=black, fill=gray!15, thick] (4,1) rectangle (5,2) node[pos=.5] {$2^e$};
        \filldraw[color=black, fill=gray!15, thick] (4,1) rectangle (4.5,0.5) node[pos=.5] {$1^e$};
        \filldraw[color=black, fill=gray!15, thick] (4,0) rectangle (4.5,0.5) node[pos=.5] {$1^e$};
        \end{tikzpicture}}%
    \caption{Diagram for configurations given in Example \ref{ex_quad_dim}.}%
    \label{pic_difflincorr}%
    \end{figure}
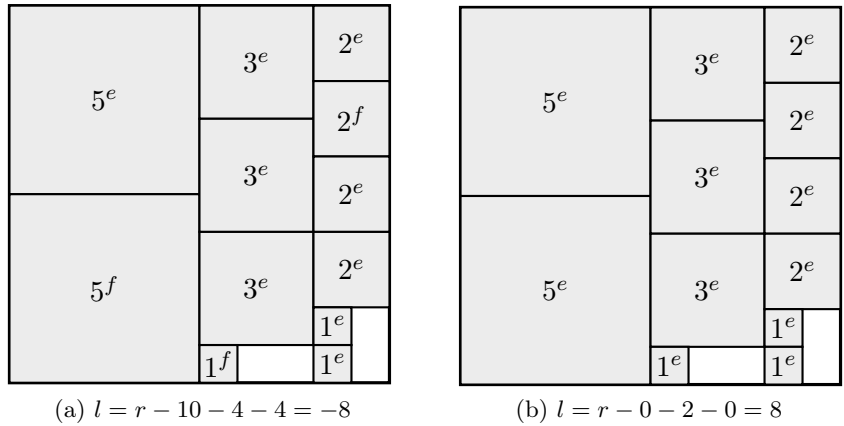
\end{example}

Using this configuration description, Simpson in Theorem 1 of \cite{simp_quadpaper} showed that there exists a bound on the rank in terms of the dimension for the non-overlapping case. He does not give his bound explicitly, but tracing through his argument gives a bound of $21252d+11173142784$, or if one more carefully applies Prop 6.2 of \cite{simp_quadpaper}, $1260d+135266208$. In this paper, we find the strictest bound on the rank of non-overlapping configurations.

\begin{theorem}\label{thm_quadthm_configsec2}(Theorem \ref{thm_quadthm_config})
    Let $d\geq 2$ be even. Any quadratic non-overlapping configuration of dimension $\Delta=2d$ has rank $r\leq 6d+36$, and $r=3\Delta+36=6d+36$ is achieved exactly by the configurations described in Example \ref{ex_mincases}.
\end{theorem}

From the above discussion, particularly Corollary 7.3 and Prop 2.5 of \cite{simp_quadpaper}, it would then follow that
\begin{theorem}\label{thm_quadthm_charvar}
    For $d\geq 2$, any numerically MC-minimal quadratic character variety on $\mathbb{P}^1\backslash \{t_1,t_2,t_3\}$ of dimension $d$ has rank $r\leq 9d+54$.
\end{theorem}

Note that we have not dealt with any issues of non-emptiness. One would need to show that for a given configuration, there is a choice of eigenvalues such that the associated quadratic character variety is nonempty for the character variety to achieve the suggested dimension.

\subsection{Middle Convolution}\label{subsec_middleconv}

For the proofs of the results in this paper, we will not explicitly need the description of middle convolution. However, the geometric objects in the background are character varieties with middle convolution applied to lower the rank. Middle convolution was defined by Katz to study rigid local systems in \cite{katz_rigid}. It was given an algebraic interpretation by Dettweiler and Reiter in \cite{dett_reiter_alg_Katz_inverGalois}. We follow Section 3 of \cite{simp_quadpaper} and Sections 3-4 of \cite{simpson_katzmiddleconv}. For the most part, we use the same notation of \cite{simp_quadpaper} outside of small changes for clarity.

Let $C=\mathbb{P}^1-\{t_1,...,t_k\}, \overline{C}=\mathbb{P}^1$ and let $D=t_1+...+t_k$. Consider $\overline{C}\times \overline{C}$ blown up at the diagonal points $(t_i,t_i)$ and call it $\overline{X}$. Let $\Delta\subset \overline{X}$ be the divisor 
$$\Delta=T+(H_1+...+H_k)+(V_1+...+V_k)+(U_1+...+U_k)$$
where 
\begin{itemize}
    \item $T$ is the strict transform of the diagonal in $\overline{C}\times\overline{C}$.
    \item $H=H_1+...+H_k$ is the strict transform of $D\times \overline{C}$ thought of as the horizontal divisors.
    \item $V=V_1+...+V_k$ is the strict transform of $\overline{C}\times D$ thought of as the vertical divisors.
    \item $U=U_1+...+U_k$ is the exceptional divisor.
\end{itemize}

Let $p,q:\overline{X}\rightarrow \overline{C}$ be the first and second projections, $j:X=\overline{X}\backslash \Delta\hookrightarrow \overline{X}$, and $p^\circ=p|_X, q^{\circ}=q|_X$ and $q'=q|_{\overline{C}\times C}$. 

\begin{definition}\label{def_middleconv}
    Let $L$ be a local system on $X$. Let $\mathcal{F}$ be a local system on $C$, then middle convolution with convolutor (or convolution kernel) $L$ is defined as
    $$MC(\mathcal{F}, L)\coloneqq R^1q_*'(j_{*}(p^{\circ,*}\mathcal{F}\otimes L)).$$
\end{definition}

We call a convolutor $L$ $MC$-invertible if there is another local system $L'$ such that $MC(MC(\mathcal{F}, L), L')\cong \mathcal{F}$. This is true for rank $1$ convolution kernels. When the convolution kernel is $MC$-invertible, middle convolution extends to an isomorphism between two character varieties on $\mathbb{P}^1-\{t_1,...,t_k\}$ changing the monodromy, i.e., the conjugacy classes $C_i$. We can understand the change explicitly when our convolution kernel is rank $1$. 

\subsubsection*{Rank 1 middle convolution}

Consider a conjugacy class in $G$ where $G=GL(r), O(r),$ or $Sp(r)$. It is determined by its conjugacy class in $GL(r)$, which in turn is determined by a partition of $r=\lambda_1+...+\lambda_k$ and $a_1, ...,a_k\in \mathbb{C}^*$. This is given as follows. The $a_i$ are the eigenvalues (note that we allow $a_i=a_j$ for $i\neq j$). For each $a_i$, let $\lambda_0(i),...,\lambda_{k(i)}(i)$ be the collection of parts of $\lambda$ associated to $a$, meaning that $\lambda_j(i)$ is the number of Jordan blocks that are size $\geq j$. This is the dual partition of the Jordan block partition. The collection of $\lambda_0(i),...,\lambda_{k(i)}(i)$ for $i=1,...,k$ form the partition $\lambda_1+...+\lambda_k$. 

Consider a rank $1$ convolution kernel $\mathcal{L}$ on $X$ and denote by $\beta:\pi_1(X,\mathbb{Z})\rightarrow \mathbb{C}^*$ its monodromy representation. It is determined by its monodromy around $H_i,V_i,U_i,T$ respectively for all $i$ which we denote $\beta^{H_i},\beta^{V_i}, \beta^{U_i}, \beta^T\in \mathbb{C}^*$ respectively. These are subject to the relations
$$\beta^{H_1}...\beta^{H_k}\beta^T=1, \beta^{V_1}...\beta^{V_k}\beta^T=1, \beta^{U_i}=\beta^{H_i}\beta^{V_i}\beta^T.$$
The last relation says that $\beta^{U_i}$ is determined by the other data. 

Now let $\alpha_i\coloneqq (\beta^{H_i})^{-1}$ be called the distinguished eigenvalue at $t_i$. Say that $j$ is distinguished at $t_i$ if the associated eigenvalue of $\lambda_j(i)$ is $\alpha_i$ and $\lambda_j(i)$ is the largest part of this eigenvalue such that $j$ is minimal with respect to all parts of the largest size. Clearly, this is unique when $\alpha_i$ is an eigenvalue at $i$. Let
$$\delta_i\coloneqq \begin{cases}
    \lambda_j(i) & \text{if }j\text{ distinguished at }i\\
    0 & \text{if }\alpha_i \text{ is not an eigenvalue at }t_i
\end{cases}$$
and define
$$\delta(\mathcal{F},\beta)\coloneqq r-\sum_{i}\delta_i.$$
Define a new partition of rank $r'=r+\delta$ as follows,
$$\lambda_j(i)'\coloneqq \begin{cases}
    \lambda_j(i) & \text{if }j\text{ is not distinguished at }i\\
    \lambda_j(i)+\delta_i & \text{if }j\text{ is distinguished at }i.
\end{cases}$$
The eigenvalue associated to the distinguished part at $t_i$ is $\beta^{V_i}$. For any non-distinguished part $j$ at $t_i$, the new associated eigenvalue is 
$$a_j(i)'=\beta^{H_i}\beta^{V_i}\beta^Ta_j(i).$$
Whenever $\beta^T\neq 1$ (this is required for middle convolution to be defined), this describes the change in the conjugacy classes.

\begin{prin}(Principle 3.4 of \cite{simp_quadpaper}, \cite{katz_rigid})
    Suppose $a_i$ are distinguished eigenvalues for an irreducible $r\geq 2$ local system $\mathcal{F}$ of multiplicities $\lambda_i$ at points $t_i$. Then $\delta=(k-2)r-\lambda_1-...-\lambda_k<0$ and for any rank $1$ middle convolution with $\beta^{H_i}=a_i^{-1}, \beta^T\neq 1$ and middle convolution is defined.
\end{prin}

\subsubsection*{Quadratic Case of middle convolution}

The above gives the description of middle convolution for the general linear group. If we want middle convolution for the quadratic case, we need to put restrictions on what kinds of convoluters are used in order to stay within quadratic local systems. A natural choice is to take a quadratic local system as our convolutor. 

\begin{prop}(Prop 3.2 in \cite{simp_quadpaper})
    The dual local systems are related by the formula
    $$MC(\mathcal{F}, \mathcal{L})^{\vee}\cong MC(\mathcal
    F^{\vee}, L^{\vee}).$$
    In particular if $\mathcal{F}$ and $L$ are quadratic then so is $MC(\mathcal{F}, \mathcal{L})$. \qed
\end{prop}

However, it is not sufficient in the quadratic case to use a rank $1$ local system to get our ``fitting in the square" property, Prop \ref{prop_quadfitsquare}. Instead, Simpson (Prop 3.6 in \cite{simp_quadpaper}) shows that when $k=3$, we can use a middle convolution with a rank $2$ quadratic convolutor, which is a composition of two rank $1$ middle convolutions with not necessarily quadratic convolutors. This allows for similar, but more involved, reductions (see Lemmas 4.1-4.5 of \cite{simp_quadpaper}) in $\delta$ to get the quadratic version of the "fitting in the square" property (Prop \ref{prop_quadfitsquare}).

\section{General Linear Bound}\label{sec_3}

Recall from Subsection \ref{subsec_GLsetup} that we consider configurations:
\begin{itemize}
    \item $r=q_1+q_2+q_3$ partition with $q_1\geq q_2\geq q_3>0$ (split the $r\times r$ square into three columns of width $q_i$)
    \item $r=\lambda_i^0+...+\lambda_{i}^{s_i}$ partitions of $r$ with $\lambda_i^j\leq q_i$ and $\lambda_i^0\geq \lambda_i^2\geq...\geq \lambda_i^{s_i}$ for $i=1,2,3$
    \item the (box) dimension is the number of empty boxes $\Delta=r^2-\sum_{i=1}^3\sum_{j=0}^{s_i}(\lambda_i^j)^2$.
\end{itemize}

We refer to 
$$\Delta_i=rq_i-\sum_{j=0}^{s_i}(\lambda_i^j)^2$$
as the column dimension. It is the number of empty boxes in the $i$th column. Note that $\Delta=\Delta_1+\Delta_2+\Delta_3$ with each $\Delta_i\geq 0$. 

\subsection{Box dimension lemmas}

We first aim to figure out the smallest dimension of a single column. Let $\mu$ and $\lambda$ be two distinct partitions of $r$. Recall that $\lambda>\mu$ in the dominance ordering if
$$\lambda_1+...+\lambda_k\geq \mu_1+...+\mu_k$$
for all $k\geq 1$ and $\lambda\neq \mu$. This defines a partial order on the set of partitions of $r$. Given $\lambda>\mu$, one can get from $\lambda$ to $\mu$ by applying a sequence of two possible moves to the Young diagram. We can either move a block down to a row of a strictly small size (first type) or take a block and make a new row (second type).
\ytableausetup{boxsize=1.25em} 
\begin{example}
    The Young diagram of the partition $(5,3,2,1)$ is 
    
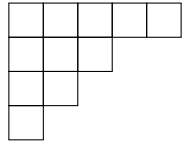
\begin{center}
    \begin{ydiagram}{5, 3,2,1}
    \end{ydiagram}
    \end{center}
    The list of the possible Young diagrams achieved by moving a block from the first row is $(4,4,2,1),(4,3,3,1),(4,3,2,2),$ and $(4,3,2,1,1)$.
    \begin{center}
        \begin{ydiagram}{4,4,2,1}
        \end{ydiagram}
        \hspace{1cm}
        \begin{ydiagram}
            {4,3,3,1}
        \end{ydiagram}
        \hspace{1cm}
        \begin{ydiagram}
            {4,3,2,2}
        \end{ydiagram}
        \hspace{1cm}
        \begin{ydiagram}
            {4,3,2,1,1}
        \end{ydiagram}
    \end{center}
\end{example}

In general, these moves change the partition $\lambda=(\lambda_0,...,\lambda_s)$ as
\begin{enumerate}
    \item (first type) $(\lambda_0,...,\lambda_{j-1},\lambda_{j}-1, \lambda_{j+1},...,\lambda_{k-1},\lambda_{k}+1,\lambda_{k+1},...,\lambda_s)$,
    \item (second type) $(\lambda_0,...,\lambda_{j-1},\lambda_{j}-1,\lambda_{j+1},...,\lambda_s,1)$.
\end{enumerate}
We will heavily use the next proposition in both the general linear case and the quadratic case (for the box dimension).

\begin{prop}\label{prop_mincoldim}
    Let $q<r$ be two natural numbers. If $\lambda>\mu$ in the dominance order of partitions of $r$ with all parts $\leq q$, then the column (box) dimension using $\mu$ is strictly greater than that of $\lambda$. In particular, the unique partition of $r$ that achieves the smallest (box) dimension for a column of size $q$ is $(q,...,q,k)$ with $k\equiv r\mod q$ and $k\in \{0,...,q-1\}$ and the dimension is given by $k(q-k)$.
\end{prop}
\begin{proof}
    Let $\lambda=(\lambda_1,...,\lambda_s)$ and $\mu=(\mu_1,...,\mu_t)$ be partitions of $r$ with all parts less than or equal to $q$. Let $\beta_{\lambda}$, resp. $\beta_{\mu}$, denote the column (box) dimension using $\lambda$, resp. $\mu$. Showing
    $$\beta_{\lambda}=r^2-\lambda_1^2-...-\lambda_s^2>r^2-\mu^{2}_1-...-\mu^2_t=\beta_{\mu}$$
    is equivalent to
    \begin{equation}\label{eq_sumsq}\lambda_1^2+...+\lambda_s^2>\mu_1^2+...+\mu_t^2.\end{equation}
    It suffices to prove Equation \ref{eq_sumsq} when $\mu$ is a partition achieved by applying one of the two moves to $\lambda$. Then it is sufficient to show
    \begin{itemize}
        \item $\lambda^2_1+...+\lambda_s^2>\lambda_1^2+...+(\lambda_j-1)^2+...+(\lambda_k+1)^2+...+\lambda_s^2$
        \item and $\lambda_1^2+...+\lambda_s^2>\lambda_1^2+...+(\lambda_j-1)^2+...+\lambda_s^2+1$.
    \end{itemize}
    The first inequality is equivalent to
    $$2\lambda_j-2\lambda_k-2>0.$$
    In order to do the first move, one must have $\lambda_j-\lambda_k\geq 2$. Hence, the above is indeed satisfied. The second move can be seen as an instance of the first, where we allow $\lambda^k=\lambda^{s+1}=0$. In that case, we find $2\lambda_j-2>0$, which holds since $\lambda_j\geq 2$ in order to apply the second move. This competes the proof that $\beta_{\lambda}>\beta_{\mu}$.

    The partition $(q,...,q,k)$ is the biggest partition with all parts $\leq q$ and hence has the smallest (box) dimension.
\end{proof}

Now, given a partition $q_1+q_2+q_3=r$, we can say what the minimal dimension is.

\begin{prop}\label{prop_mincolconfig}
    Let $q_1+q_2+q_3=r$ be a partition of $r$. Then the smallest nonzero dimension with column widths $q_i$ is 
    $$k_1(q_1-k_1)+k_2(q_2-k_2)+k_3(q_3-k_3)$$
    where $k_i\equiv r\mod q$, $k_i\in \{0,...,q_i-1\}$ unless all $k_i=0$ in which case the minimal nonzero dimension is given by $2q_i-2$ for the smallest $q_i$ that is not $1$.
\end{prop}
\begin{proof}
    If at least one of the $k_i$ is nonzero, then by Prop \ref{prop_mincoldim}, the smallest box dimension at each of the columns will work. If all $k_i=0$, then the smallest dimension is zero. Let $q_i$ be the smallest of the $q_i$ that are not $1$. Then the next biggest partition in the dominance ordering will be $(q_i,...,q_i,q_i-1,1)$.
\end{proof}

\subsection{General Linear Bound}

We have reduced finding the minimal dimension to finding the optimal partition $q_1+q_2+q_3$ of $r$. Now we prove the theorem via casework. When $\Delta=0$, the dimension of the associated character variety is $2$. 
In order for $\Delta$ to be zero, we need each $q_i$ to divide $r$, i.e $q_1=\frac{r}{m_1},q_2=\frac{r}{m_2},$ and $q_3=\frac{r}{m_3}$. Hence, we just need to solve 
$$\frac{1}{m_1}+\frac{1}{m_2}+\frac{1}{m_3}=1.$$
It follows that the dimension of the character variety is $2$ if and only if $(q_1,q_2,q_3)$ is one of $\left(\frac{r}{3},\frac{r}{3},\frac{r}{3}\right)$, $\left(\frac{r}{2},\frac{r}{4},\frac{r}{4}\right)$, and $\left(\frac{r}{2},\frac{r}{3},\frac{r}{6}\right)$ (see also, Lemma 2.13 of \cite{simpson_katzmiddleconv}). Hence, dimension $2$ MC-minimal character varieties can occur for arbitrarily high rank. Now we consider character varieties of bigger dimensions on three points and find the bound. The proof of the following can be thought of as a mini-version of our proof in the quadratic case, Theorem \ref{thm_quadthm_charvar}. For each $i$, write $r=n_iq_i+c_i$ where $\frac{-q_i}{2}< c_i\leq \frac{q_i}{2}$. By Proposition \ref{prop_mincoldim},
$$\Delta\geq \Delta_i\geq |c_i|(q_i-|c_i|).$$
Note that $q_i=\frac{r}{n_i}-\frac{c_i}{n_i}$. Think of $c_i$ as a measure of how close $q_i$ is to dividing $r$. 
\begin{theorem}\label{thm_GLrbound} (Theorem \ref{thm_GLrboundsec2})
    Given $\Delta>0$ even, any minimal configuration of dimension $\Delta$ has rank $r\leq 3\Delta+6$ and $r=3\Delta+6$ is achieved for exactly the configuration with $(q_1,q_2,q_3)=\left(\frac{r}{2}, \frac{r}{3}, \frac{r}{6}\right)$ and columns $\left(\frac{r}{2},\frac{r}{2}\right),\left(\frac{r}{3},\frac{r}{3},\frac{r}{3}\right),$ and $\left(\frac{r}{6},\frac{r}{6},\frac{r}{6},\frac{r}{6},\frac{r}{6},\frac{r}{6}-1,1\right)$ (see Figure \ref{thm_GLrboundsec2}).
\end{theorem}
\begin{proof}
    We prove the theorem by showing that for any configuration, $\Delta\geq \frac{r}{3}-2$ and that equality is only achieved for the configuration given in the statement. Assume we are given a configuration with $\Delta\leq \frac{r}{3}-2$. 
    
    We first focus on the first column and show that it must have either $q_1=\frac{r}{2}$ or $q_1=\frac{r}{3}$ by first reducing $n_1$ to either $2$ or $3$, then showing $c_1=0$. Suppose $n_1=1$, then $r=q_1+c_1\leq \frac{3}{2}q_1$ with $c_1=q_2+q_3\geq 2$. Then
    $$\Delta\geq c_1(q_1-c_1)\geq c_1\frac{q_1}{2}.$$
    Since $r\geq \frac{3}{2}q_1$,
    $$\Delta \geq \frac{c_1 r}{3}\geq \frac{r}{3}>\frac{r}{3}-2.$$
    If $n_1\geq 4$, then
    $$r=n_1q_1+c_1\geq 4q_1-\frac{q_1}{2}=\frac{7q_1}{2}>3q_1$$
    which contradicts our order assumption of the $q_i$. Hence, $n_1$ must be either $2$ or $3$. Now we claim that $c_1=0$. 
    
    Suppose $n_1=2$, then $r= 2q_1+c_1\leq \frac{5}{2}q_1$ and $\Delta\geq |c_1|\frac{q_1}{2}\geq |c_1|\frac{r}{5}.$
    If $|c_1|\geq 2$, then this is greater than $\frac{r}{3}-2$. If $c_1=1$ or $c_1=-1$, then 
    $$\Delta\geq q_1-1\geq \frac{r}{2}-\frac{1}{2}-1>\frac{r}{3}-2.$$
    Thus if $n_1=2$, then $c_1=0$.

    Suppose $n_1=3$. By the order assumption $c_1\leq 0$. Then $r=3q_1+c_1\leq \frac{7}{2}q_1$ and $\Delta\geq |c_1|\frac{r}{7}.$
    If $c_1\leq -3$, we would be done. If $c_1=-2$, then $\Delta\geq 2(q_1-2)=\frac{2r}{3}+\frac{4}{3}-2$ and if $c_1=-1$, $\Delta\geq q_1-1=\frac{r}{3}-\frac{2}{3}.$ Either choice is bigger than $\frac{r}{3}-2$. 

    Now, we must have $q_1=\frac{r}{2}$ or $q_1=\frac{r}{3}$. However, if $q_1=\frac{r}{3}$, then $q_2=q_3=\frac{r}{3}$ by the order assumption. By Proposition \ref{prop_mincoldim}, $\Delta=\frac{2r}{3}-2$.  
    
    Hence, we can assume $q_1=\frac{r}{2}$. The order assumption implies that $q_2\geq \frac{r}{4}$. We now claim that $n_2\in \{3,4\}$. If $n_2=1$, then $r=q_2+c_2\leq \frac{3}{2}q_2$ which contradicts $q_2\leq \frac{r}{2}$. If $n_2=2$, the same computations as in the case of $q_1$ will imply that $c_2=0$, but this is not possible since $q_3\geq 1$. If $n_2\geq 5$, then $r=n_2q_2+c_2\geq 5q_2-\frac{q_2}{2}=\frac{9}{2}q_2$ which contradicts $q_2\geq \frac{r}{4}$. Now we have either $n_2=3$ or $n_2=4$.

    \textbf{Case:} $n_2=4$. In this case, $q_1=\frac{r}{2}, q_2=\frac{r}{4}-\frac{c_2}{4},$ and $q_3=\frac{r}{4}+\frac{c_2}{4}$. Our order assumption implies $c_2\leq 0$. Notice that $\Delta_2\geq |c_2|(q_2-|c_2|)\geq |c_2|\frac{q_2}{2}\geq |c_2|\frac{r}{9}.$
    If $|c_2|\geq 3$, this is greater than $\frac{r}{3}-2$. If $c_2=-2$, then
    $$\Delta_2=2\left(\frac{r}{4}+\frac{1}{2}-2\right)=\frac{r}{2}-3.$$
    We cannot have $c_2=-1$ since $r$ is even. Lastly, if $c_2=0$, then $q_1=\frac{r}{2}, q_2=q_3=\frac{r}{4}$ and Prop \ref{prop_mincoldim} gives that $\Delta=\frac{r}{2}-2$.

    \textbf{Case:} $n_2=3$. In this case, $q_1=\frac{r}{2}, q_2=\frac{r}{3}-\frac{c_2}{3},$ and $q_3=\frac{r}{6}+\frac{c_2}{3}$. Similarly to the $n_3=4$ case, we have that $\Delta_2\geq |c_2|\frac{r}{7}$
    so if $|c_2|\geq 3$, we are done. If $c_2=2$, then
    $$\Delta_2=\frac{2r}{3}-\frac{4}{3}-4$$
    which is greater than $\frac{r}{3}-2$ once $r\geq 10$. Note that it is only possible to have $c_2=2$ when $q_2\geq 4$, and so such configurations only show up for $r\geq 12$. If $c_2=-2$, then
    $$\Delta_2=\frac{2r}{3}+\frac{4}{3}-4$$
    which is greater than $\frac{r}{3}-12$ for $r\geq 2$. If $|c_2|=1$, then
    $$\Delta_2\geq \frac{r}{3}-\frac{1}{3}-1>\frac{r}{3}-2.$$
    Hence, we must have $c_2=0$. Then $q_1=\frac{r}{2}, q_2=\frac{r}{3},$ and $q_3=\frac{r}{6}$ with $\Delta=\frac{r}{3}-2$ as desired.
\end{proof}

Combining the above with the discussion in Section \ref{subsec_GLsetup} gives us Theorem \ref{thm_GLrbound_charvarversion}.

\section{Quadratic Bound Reductions}\label{sec_4}

We now work towards proving the quadratic bound, i.e., Theorem \ref{thm_quadthm_configsec2}. In this section, we do our major reductions, reducing the possible configurations with small dimension with respect to the rank to a finite family of cases, which we deal with in the next section. Recall from Section \ref{sec_2} that a (non-overlapping) numerically MC-minimal quadratic configuration of rank $r$ is given by
\begin{itemize}
    \item $r=q_1+q_2+q_3$ with $q_i\geq 1$ and $q_1\geq q_2\geq q_3$ (split $r\times r$ square into three columns of width $q_i$).
    \item $\sum_{j=0}^{s_i} \lambda_i^j=r$ with $1\leq \lambda_i^{j}\leq q_i$ for all $j$ and $i=1,2,3$.
    \item Each $\lambda_i^j$ is assigned type $e$, type $f$, or type $m$. The type $m$ parts appear in pairs. Then let $e_i^0\geq e_i^1\geq ...\geq e_i^{s_{i,e}}$, $f_i^0\geq f_i^1\geq...\geq f_i^{s_{i,f}}$, $m_i^0\geq m_i^1\geq...\geq m_i^{s_{i,m}}$ denote the ordered set of type $e$, type $f$, and type $m$ parts at column $i$.
    \item The box dimension is $\beta=r^2-\sum_{i=1}^3\sum_{j=0}^{s_i}(\lambda_i^j)^2.$
    \item The linear correction is $l=r-\sum_{i,j}(-1)^{j}e_i^j-\sum_{i,j}(-1)^{j}f_i^j$.
    \item The dimension of the configuration is $\Delta=\beta-\varepsilon l$ where $\varepsilon=\pm 1$.
\end{itemize}

The first observation we make is that we can replace any pair of type $m$ squares with a pair of type $e$ without changing the dimension.

\begin{lemma}\label{lem_typemtotypee}
    For any non-overlapping numerically MC-minimal quadratic configuration, there exists a non-overlapping MC-minimal quadratic configuration with the same dimension with no type $m$ squares.
\end{lemma}
\begin{proof}
    Consider the dimension at column $i$ with a pair of type $m$ squares. The box dimension does not change based on the type, so we look at the linear correction. Let $e_i^0\geq e_i^1\geq ...\geq e_i^{s_i,e}$. Suppose we change the pair of $m$ squares of size $m_i^j$ into two type $e$ squares. Then we would have $e_i^0\geq e_i^1\geq ...\geq e_i^{k}\geq m_i^j\geq m_i^{j}\geq e_i^{k+1}\geq ...\geq e_i^{s_{i,e}}$. Then the contribution to the linear correction of the type $e$ squares is
    $$\sum_{j=0}^{k}(-1)^{j}e_i^j+m_i^j-m_i^j+\sum_{j=k+1}^{s_{i,e}}(-1)^{j+2}e_i^j=\sum_{j=0}^{s_{i,e}}(-1)^je_i^j.$$
    Thus, the linear correction would not change, and the dimension would be the same.
\end{proof}

From now on, we only consider configurations with squares of type $e$ or $f$. We refer to 
\begin{itemize}
    \item $\beta_i=rq_i-\sum_{j=0}^{s_i}(\lambda_i^j)^2$ as the $i$th column box dimension.
    \item $l_i=q_i-\sum_{j=0}^{s_{i,e}}(-1)^je_i^j-\sum_{j=0}^{s_{i,f}}(-1)^jf_i^j$ as the linear correction of the $i$th column.
    \item $\Delta_i=\beta_i-\varepsilon l_i$ as the $i$th column dimension of the configuration.
\end{itemize}

Clearly $\beta_1+\beta_2+\beta_3=\beta$, $l_1+l_2+l_3=l$ and $\Delta_1+\Delta_2+\Delta_3=\Delta$.

\subsection{Quadratic Column Dimension}\label{subsec_quadcoldim}

First, we need some control over the linear correction.
\begin{lemma}\label{lem_lincorrbound}(Lemma 5.2 in \cite{simp_quadpaper} for columns)
    Let $l_i$ be the linear correction at the $i$th column and $l$ be the linear correction. Then $-q_i\leq l_i\leq q_i$ and $-r\leq l\leq r$.
\end{lemma}
\begin{proof}
    By the order assumption $e_i^0\geq e_i^1\geq ...\geq e_i^{s_{i,e}}$ and $f_i^0\geq f_i^1\geq ...\geq f_i^{s_{i,f}}$,
    $$\sum_{j=0}^{s_{i,e}}(-1)^{j}e_i^j\leq e_i^0\text{ and } \sum_{j=0}^{s_{i,f}}(-1)^{j}f_i^j\leq f_i^0.$$
    Since $e_i^0, f_i^0\leq q_i$, $0\leq \sum_{j=0}^{s_{i,e}}(-1)^{j}e_i^j
    +\sum_{j=0}^{s_{i,f}}(-1)^{j}f_i^j\leq 2q_i$ and hence $-q_i\leq l_i\leq q_i$.
\end{proof}

We will use Lemma \ref{lem_lincorrbound} to say that if $\beta\geq r+D$, then $\Delta\geq D$ or if $\beta_i\geq q_i+D_i$, then $\Delta_i\geq D_i$. Unlike for the $GL(r)$-case, the column dimension can be negative even if the total dimension $\Delta$ is positive. Our approach will be similar in that we will still figure out what the minimal dimension of a column is. However, we will be forced to consider the next smallest column dimensions.

\begin{example}
    A column $\left(\frac{r}{2}, \frac{r}{2}\right)$ has linear correction $l_q=\frac{r}{2}$ or $\frac{-r}{2}$. A column $\left(\frac{r}{3},\frac{r}{3},\frac{r}{3}\right)$ always has $l_q=0$. Both of these have box dimension zero.
    
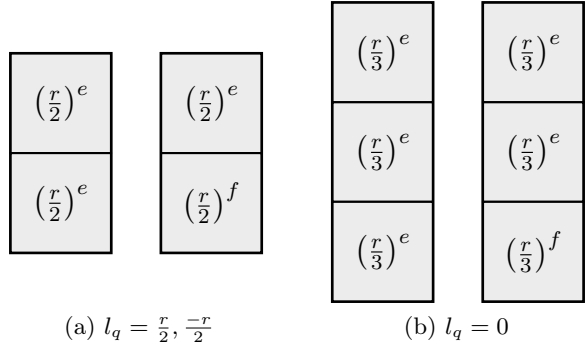
\begin{figure}[h]
    \centering
    \subfloat[\centering $l_q=\frac{r}{2}, \frac{-r}{2}$]{{\begin{tikzpicture}[scale=0.66]
        \filldraw[color=white] (0,0) rectangle (1,1); 
        \draw[very thick] (0,1) rectangle (2,5);
        \filldraw[color=black, fill=gray!15, thick] (0,3) rectangle (2,5) node[pos=.5] {$\left(\frac{r}{2}\right)^e$};%
        \filldraw[color=black, fill=gray!15, thick] (0,1) rectangle (2,3) node[pos=.5] {$\left(\frac{r}{2}\right)^e$};%
        \draw[very thick] (3,1) rectangle (5,5);
        \filldraw[color=black, fill=gray!15, thick] (3,3) rectangle (5,5) node[pos=.5] {$\left(\frac{r}{2}\right)^e$};%
        \filldraw[color=black, fill=gray!15, thick] (3,1) rectangle (5,3) node[pos=.5] {$\left(\frac{r}{2}\right)^f$};%
        \end{tikzpicture}
        }}%
    \qquad
    \subfloat[\centering $l_q=0$]{\begin{tikzpicture}[scale=0.66]
        \draw[very thick] (0,0) rectangle (2,6);
        \filldraw[color=black, fill=gray!15, thick] (0,6) rectangle (2,4) node[pos=.5] {$\left(\frac{r}{3}\right)^e$};%
        \filldraw[color=black, fill=gray!15, thick] (0,4) rectangle (2,2) node[pos=.5] {$\left(\frac{r}{3}\right)^e$};%
        \filldraw[color=black, fill=gray!15, thick] (0,0) rectangle (2,2) node[pos=.5] {$\left(\frac{r}{3}\right)^e$};%
        \draw[very thick] (3,0) rectangle (5,6);
        \filldraw[color=black, fill=gray!15, thick] (3,6) rectangle (5,4) node[pos=.5] {$\left(\frac{r}{3}\right)^e$};%
        \filldraw[color=black, fill=gray!15, thick] (3,4) rectangle (5,2) node[pos=.5] {$\left(\frac{r}{3}\right)^e$};%
        \filldraw[color=black, fill=gray!15, thick] (3,2) rectangle (5,0) node[pos=.5] {$\left(\frac{r}{3}\right)^f$};%
        \end{tikzpicture}}%
    \end{figure}
\end{example}

Roughly, a column partition that looks almost like an even number of $q$ boxes will have linear correction around $\pm q$, and if odd, then around $0$. In the following lemma, we find all possible choices of linear correction when the number of parts is odd. Note that this gives all possible choices of linear correction when the number of parts is even by adding a size zero part. Then, for the sake of making computation easier in Section \ref{sec_5}, we use this to find all possible column dimensions when there is at least one part of size $q$. Notice that the result shows that for such partitions, the choice of $\varepsilon$ does not matter.

\begin{lemma}\label{lem_colqq_qqa}
    Suppose we have a partition $(\lambda^0,...,\lambda^s)$ with $\lambda^0\geq ...\lambda^s\geq 0$ and $s$ even. Then the possible choices of linear correction are $q-\lambda^0-\sum_{j=1}^{s/2}\pm (\lambda^j-\lambda^{j+1})$. In particular, if given a column of the form $(q,...,q,a_1,...a_t)$ with $q\geq a_1\geq ...\geq a_t\geq 0$, at least one part of size $q$, and $t$ such that the number of parts is odd, then the possible column dimensions are as follows. 
    \begin{itemize}
        \item If there is an even number of $q$'s, then the possible column dimensions are $\beta_q\pm (q-a_1)\pm (a_2-a_3)\pm ...\pm (a_{t-1}-a_{t})$
        \item If there is an odd number of $q$'s, then the possible column dimensions are $\beta_q\pm (a_1-a_2)\pm (a_3-a_4)\pm ...\pm (a_{t-1}-a_t)$.
    \end{itemize}
    where $\beta_q=\sum_{j=1}^ta_j(q-a_j)$ is the column box dimension.
\end{lemma}
\begin{proof}
    For this proof, we drop our assumption that $e^0\geq f^0$ and swap the two types as convenient. Refer to $\pm \lambda^j$ with the correct sign to fit into the linear correction sum as the contribution to the linear correction of a part $\lambda^j$. Write $\lambda^{j,e}$ or $\lambda^{j,f}$ when $\lambda^j$ has been given type $e$ or $f$ respectively. Let $s_{e,>\lambda^j}$ denote the number of type $e$ parts above part $\lambda^j$ and $s_{f,>a\lambda^j}$ the number of type $f$ parts above part $\lambda^j$. 

    Suppose one of $s_{e,>\lambda^j}$ and $s_{f,>\lambda^j}$ is odd and the other is even. Up to swapping type $e$ and $f$, assume $s_{e,>\lambda^j}$ is odd. The choices of contribution for $\lambda^j$ and $\lambda^{j+1}$ are $-\lambda^{j,e}+\lambda^{j+1,e}, -\lambda^{j,e}+\lambda^{j+1,f}, +\lambda^{j,f}-\lambda^{j+1,e},$ and $ +\lambda^{j,f}-\lambda^{j+1,f}$. Now $s_{e,>\lambda^{j+2}}$ and $s_{f,>\lambda^{j+2}}$ satisfy again that one is odd and one is even. By induction starting with $j=1$, the possible choices of linear correction are then $q-\lambda^0-\sum_{j=1}^{s/2}\pm (\lambda_j-\lambda_{j+1})$. 

    To find the possible column dimensions of $(q,...,q,a_1,...a_t)$ in the case where the number of $q$ is odd, apply the induction starting at $a_1$ and note that the contribution to the linear correction of $(q,...,q)$ is always $-q$. The possible choices of linear correction are then 
    $$q-q\pm (a_1-a_2)\pm ...\pm (a_{t-1}-a_{t}).$$ For the even number of $q$ case, apply the induction starting at $a_2$ to get the possible linear corrections for $(q,...,q,a_1,...,a_t)$. The contribution of $(q,...,q,a_1)$ is either $-a_1$ if $s_{e,>a_1}$ is even or $-2q+a_1$ if $s_{f,>a_1}$ is odd. The possible contribution of $a_2,...,a_t$ is $\pm (a_2-a_3)\pm ...\pm (a_{t-1}-a_t)$. Then the possible linear corrections for $(q,...,q,a_1,...,a_t)$ are
    $$q-a_1\pm (a_2-a_3)\pm ...\pm (a_{t-1}-a_t), q-2q+a_1\pm (a_2-a_3)\pm...\pm (a_{t-1}-a_t).$$
\end{proof}

Just as in the general linear case, we want to find the minimal column dimension and when exactly it is achieved. We first prove an analog of Prop \ref{prop_mincoldim}. Recall the discussion before Prop \ref{prop_mincoldim}, that we can get from any partition of $r$ to any lower one in the dominance ordering using one of two moves to the Young diagram. 

We refer to the first one as a move of the first type and the second as a move of the second type. Recall that this can be written as changing $(\lambda^0,...,\lambda^s)$ respectively to
\begin{enumerate}
    \item (first type) $(\lambda^0,...,\lambda^{j-1},\lambda^j-1, \lambda^{j+1},...,\lambda^{k-1},\lambda^{k}+1,\lambda^{k+1},...,\lambda^s)$,
    \item (second type) $(\lambda^0,...,\lambda^{j-1},\lambda^{j}-1,\lambda^{j+1},...,\lambda^s,1)$.
\end{enumerate}

\begin{prop}\label{prop_quadDelta_domorder}
    Let $\mu<\lambda$ be partitions of $r$ with all parts $\leq q$ and let $\Delta_{\lambda}, \Delta_{\mu}$ be the minimal dimensions possible using the partitions $\lambda, \mu$ respectively. Then $\Delta_{\lambda}\leq \Delta_{\mu}$. In particular, if $\mu$ arises by applying one move to the Young diagram, then (using the notation above)
    \begin{enumerate}
        \item for a move of the first type, equality is only possible if $\lambda^j-\lambda^k=2$.
        \item for a move of the second type, equality is only possible if $\lambda^j=2$.
    \end{enumerate}
\end{prop}
\begin{proof}
    Since any such $\mu$ can arise from a sequence of the two moves applied to $\lambda$, it is sufficient to show $\Delta_{\lambda}\leq \Delta_{\mu}$ when $\mu$ arises by applying one move. 

    Suppose $\mu$ arises from a move of the first type, i.e., 
    $$\mu=(\lambda^0,...,\lambda^{j-1}, \lambda^j-1,\lambda^{j+1},...,\lambda^{k-1}, \lambda^k+1,\lambda^{k+1},...,\lambda^s).$$ Consider any choice of types $e,f$ for the parts of $\mu$. Make the same choices for $\lambda$, i.e., the type of $\lambda^x$ is the same type as $\mu^x$. Write the linear correction $l_{\lambda}=q-\sum_{x} (-1)^{\varepsilon_x}\lambda^x$ where $\varepsilon_x$ is $\pm 1$ according to the type and parity of the ordering among its type. Then
    $$l_{\mu}=q-\sum_{x\neq j,k}(-1)^{\varepsilon_x}\lambda^x-(-1)^{\varepsilon_j}(\lambda^j-1)-(-1)^{\varepsilon_k}(\lambda^k+1)=l_{\lambda}+(-1)^{\varepsilon_j}-(-1)^{\varepsilon_k}$$
    and it follows that $|l_{\mu}-l_{\lambda}|\leq 2$. We can view the second type as a move of the first type with $k=s+1$, where we allow $\lambda^{s+1}=0$. Then for either type of move move, $|l_{\mu}-l_{\lambda}|\leq 2$ where $\mu$ is achieved from $\lambda$ by applying one move.

    Now, consider the box dimension. For the first kind of move,
    $$\beta_{\mu}-\beta_{\lambda}=2\lambda_j-2\lambda_k-2\geq 2$$
    since $\lambda_j-\lambda_k\geq 2$. Since the linear correction differs by at most $2$, it follows from $\Delta_{\mu}\geq \Delta_{\lambda}$ with equality only possible if $\lambda_j-\lambda_k=2$. 

    For the second kind of move, 
    $$\beta_{\mu}-\beta_{\lambda}=2\lambda_j-2\geq 2$$
    since $\lambda_j\geq 2$. It follows that $\Delta_{\mu}\geq \Delta_{\lambda}$ with equality only possible if $\lambda_j=2$.
\end{proof}

\begin{prop}\label{prop_quadmincoldim}
    Let $r=mq+k$ where $k\equiv r\mod q$ and $k\in \{0,...,q-1\}$. Then the minimal column dimension occurs at $(q,...,q,k)$ and is given by
    \begin{itemize}
        \item if $m$ is even, $k(q-k)-q+k$
        \item if $m$ is odd, $k(q-k)-k$
    \end{itemize}
    unless 
    \begin{itemize}
        \item $k=2$ and $m$ is even, in which case the minimal column dimension also occurs at $(q,...,q,1,1)$.
        \item $k=q-2$ and $m$ is odd, in which case the minimal dimension also occurs at $(q,...,q,q-1,q-1)$.
    \end{itemize}
\end{prop}
\begin{proof}
    Let $\Delta_0$ be the proposed minimal column dimension, i.e., $k(q-k)-q+k$ if $m$ is even, and $k(q-k)-k$ if $m$ is odd. Let $\Delta_{\mu}$ be a dimension of the column partition $\mu$ we are considering. By Lemma \ref{lem_colqq_qqa}, the partition $(q,...,q,k)$ has possible dimensions $k(q-k)\pm (q-k)$ if $m$ is even, and possible dimensions $k(q-k)\pm k$ if $m$ is odd. Then $\Delta_0$ is the minimal column dimension using the partition $(q,...,q,k)$. Now consider partitions coming from applying one move to the Young diagram. By Prop \ref{prop_quadDelta_domorder}, $\Delta_{\mu}>\Delta_0$ for each of these partitions except possibly if either $k=q-2$ and $\mu=(q,..,q,q-1,q-1)$, or $k=2$ and $\mu=(q,...,q,1,1)$.
    
    Let $m$ be odd with $k=q-2$ and $\mu=(q,...,q,q-1,q-1)$. By Lemma \ref{lem_colqq_qqa}, $\Delta_{\mu}\in \{q-2, q, 3q-2,3q-4\}$. Thus, the minimal dimension $\Delta_0=q-2$ is achieved with this $\mu$. The only partitions that are achieved from applying one move to $\mu$ are $\mu'=(q,...,q,q-1,q-2,1)$ and $\mu''=(q,...,q,q-1,q-1,q-1,1)$. By Prop \ref{prop_quadDelta_domorder}, $\Delta_{\mu'}>\Delta_{\mu}$ unless possibly if $q-1=2$ and $\Delta_{\mu''}>\Delta_{\mu}$ unless possibly if $q=2$. First consider $\mu'=(3,...,3,2,1,1)$. By Lemma \ref{lem_colqq_qqa}, $\Delta_{\mu'}=6\pm 1\in \{5,7\}$ which is greater than $\Delta_0=1$. Next consider $\mu''=(2,...,2,1,1,1,1)$ which has only one possible dimension $\Delta_{\mu''}=4$ which is greater than $\Delta_0=0$. Hence, any partition less than $(q,...,q,q-1,q-1)$ will have column dimension bigger than $\Delta_0$.

    Again, let $m$ be odd. Let $k=2$ and consider the partition $\mu=(q,...,q,1,1)$. Then by Lemma \ref{lem_colqq_qqa}, $\Delta_{\mu}=\beta_{\mu}=2q-2>\Delta_0=2q-6$. 

    Let $m$ be even with $k=q-2$ and consider $\mu=(q,...,q,q-1,q-1)$. Then by Lemma \ref{lem_colqq_qqa}, $\Delta_{\mu}=2q-2$ which is greater than $\Delta_0=2q-6$. 

    Again, let $m$ be even. Let $k=2$ and consider the partition $\mu=(q,...,q,1,1)$. Here, $\Delta_0=q-2$. The possible dimensions of $\mu$ are $\Delta_{\mu}\in \{q-2,q,3q-4,3q-2\}$ and hence achieve $\Delta_0$. The only partitions that are achieved by applying one move to $\mu$ are $\mu'=(q,...,q,q-1,2,1)$ and $\mu''=(q,...,q,q-1,1,1,1)$. By Prop \ref{prop_quadDelta_domorder}, $\Delta_{\mu'}>\Delta_{\mu}$ unless possibly if $q-1=2$ and $\Delta_{\mu''}>\Delta_{\mu}$ unless possibly if $q=2$. Consider $\mu'=(3,...,3,2,2,1)$. By Lemma \ref{lem_colqq_qqa}, $\Delta_{\mu'}=6\pm 1\in \{5,7\}$ which is greater than $\Delta_0=1$. Consider $\mu''=(2,...,2,1,1,1,1)$. By Lemma \ref{lem_colqq_qqa}, $\Delta_{\mu''}=4\pm 1\pm 1\in \{2, 4, 6\}$ which is greater than $\Delta_0=0$. Thus, any partition less than $(q,...,q,1,1)$ has bigger column dimension than $\Delta_0$.
\end{proof}

We can now find exactly when a column has negative dimension. In particular, due to the order assumption $q_1\geq q_2\geq q_3$, the following lemma implies that $\Delta_1<0$ only if $q_1=\frac{r}{2}$, and $\Delta_i\geq \frac{-r}{4}$ for $i=2,3$. 
\begin{lemma}\label{lem_negcol}
    The column dimension is negative if and only if $r=mq$ for $m$ even with partition $(q,...,q)$ and column dimension $-q=\frac{-r}{m}$.
\end{lemma}
\begin{proof}
    Write $r=mq+k$. Assume $m$ is odd. Then by Prop \ref{prop_quadmincoldim}, $k(q-k)-k\leq \Delta<0$ implies $(q-k)\leq 0$, which is not possible.

    Now assume $m$ is even. Then
    $$k(q-k)-q+k=(k-1)(q-k)<0$$
    only if $k-1< 0$ which can only occur if $k=0$. By Lemma \ref{lem_colqq_qqa}, the partition $(q,...,q)$ can have column dimension $-q$. Any other partition will be lower in the dominance ordering than $(q,...,q,q-1,1)$. By Lemma \ref{lem_colqq_qqa}, the column dimension using $(q,...,q,q-1,1)$ is greater than $q>0$. The result then follows from Prop \ref{prop_quadDelta_domorder}.
\end{proof}

\subsection{Quadratic Reductions}\label{subsec_quadbound}

We will claim for the non-overlapping case that the largest rank $r$ to achieve dimension $\Delta=2d$ is $r=6d+36=3\Delta+36$ or equivalently $\Delta=\frac{r}{3}-12$, and that it is achieved by exactly the following configurations.

\begin{example}\label{ex_mincases}\textbf{(Claimed minimal dimension non-overlapping configurations)}
    The configurations are listed with only one possibility for their linear correction per $\varepsilon$, but there are multiple choices of types that produce the same linear correction and so the dimension. We list the three column partitions where $\lambda^e$ and $\lambda^f$ mean that the square of size $\lambda$ has type $e$ or $f$ respectively. Let $q_1=\frac{r}{2}, q_2=\frac{r}{3}$, and $q_3=\frac{r}{6}$. The configurations are
    \begin{itemize}
        \item $\varepsilon=-1$, $(q_1^e,q_1^f),(q_2^e,q_2^e,q_2^e),(q_3^e,q_3^e,q_3^e,q_3^e,q_3^e,(q_3-2)^e,2^e)$.
        \item $\varepsilon=1$, $(q_1^e,q_1^e),(q_2^e,q_2^e,q_2^e),(q_3^e,q_3^e,q_3^e,q_3^e,q_3^e,(q_3-2)^f,2^e)$.
        \item $\varepsilon=-1$, $(q_1^e,q_1^f), (q_2^e,q_2^e,q_2^e), (q_3^e,q_3^e,q_3^e,q_3^e, q_3^e,(q_3-3)^f,3^e)$.
        \item $\varepsilon=1$, $(q_1^e,q_1^e), (q_2^e,q_2^e,q_2^e), (q_3^e,q_3^e,q_3^e,q_3^e, q_3^e,(q_3-3)^e,3^e)$.
    \end{itemize}
    In each of these, the column dimensions are $\Delta_1=\frac{-r}{2}, \Delta_2=0,$ and $\Delta_3=\frac{5r}{6}-12$ and hence $\Delta=\frac{r}{3}-12$. We show the calculations for the first two. Observe that $\beta_1=0, \beta_2=0$, and $\beta_3=4q_3-8$. For the first column, $(q_1^e,q_1^f)$ has linear correction $q_1-q_1^e-q_1^f=-q_1=\frac{-r}{2}$ and $(q_1^e,q_1^e)$ has linear correction $q_1-q_1^e+q_1^e=q_1=\frac{r}{2}$. Then $\Delta_1=\beta_1-\varepsilon l_1=\frac{-r}{2}$. For the second column, $(q_2^e,q_2^e,q_2^e)$ has linear correction $q_2-q_2^e+q_2^e-q_2^e=0$ so $\Delta_2=0$. For the third column, $(q_3^e,q_3^e,q_3^e,q_3^e,q_3^e,(q_3-2)^e,2^e)$ has linear correction
    $$l_3=q_3-q_3^e+q_3^e-q_3^e+q_3^e-q_3^e+(q_3-2)^e-2^e=q_3-4$$
    while
    $(q_3^e,q_3^e,q_3^e,q_3^e,q_3^e,(q_3-2)^f,2^e)$ has linear correction
    $$l_3=q_3-q_3^e+q_3^e-q_3^e+q_3^e-q_3^e+2^e-(q_3-2)^f=-q_3+4.$$
    Hence $\Delta_3=\beta_3-\varepsilon l_3=4q_3-8+q_3-4=\frac{5r}{6}-12$.
\end{example}

In the next two sections, we work to prove:
\begin{theorem}\label{thm_quadthm_config}(Theorem \ref{thm_quadthm_configsec2})
    Let $d\geq 2$ be even. Any quadratic non-overlapping configuration of dimension $\Delta=2d$ has rank $r\leq 6d+36$, and $r=3\Delta+12=6d+36$ is achieved exactly by the configurations described in Example \ref{ex_mincases}.
\end{theorem}

Note that since $d$ is assumed even, we only need to be concerned with $\Delta$ that are divisible by $4$. In particular, the smallest configuration dimension we consider is $\Delta=4$. By Example \ref{ex_mincases}, we will only be concerned with rank $r\geq 48$.

Let $c_i\equiv r\mod q$ where $c_i\in (-\frac{q_i}{2}, \frac{q_i}{2}]$. This means if $c_i\geq 0$, then $c_i=k_i$ and if $c_i<0$, then $k_i=c_i+q$. Write
$$r=n_iq+c_i.$$
As mentioned before Theorem \ref{thm_GLrbound}, we think of $c_i$ as a measure of how close $q_i$ is to dividing $r$. The next couple results will show that if $\Delta\leq \frac{r}{3}-12$, then $(n_1,n_2)\in \{(2,2), (2,3),(2,4),(3,3)\}$ and $|c_i|\leq n_i$. Dealing with these cases is the content of Section \ref{sec_5}. The condition that $|c_i|\leq n_i$ is equivalent to $\frac{r}{n_i}-1\leq q_i\leq 
\frac{r}{n_i}+1$. One may notice from the proofs of the following two Lemmas that when we do not have such $n_i$ and $c_i$, the dimension tends to be much bigger than $\frac{r}{3}-12$.

Observe that by Lemma \ref{lem_colqq_qqa},
\begin{itemize}
    \item if $n_i$ is even, the minimal column dimension is $|c_i|(q_i-|c_i|)-q_i+|c_i|$.
    \item if $n_i$ is odd, the minimal column dimension is $|c_i|(q_i-|c_i|)-|c_i|$.
\end{itemize}

By Lemma \ref{lem_lincorrbound}, to show $\Delta>\frac{r}{3}-12$, it is sufficient to show that $\beta> \frac{4r}{3}-12$.

\begin{lemma}\label{lem_q1c1res}
    If $\Delta\leq  \frac{r}{3}-12$ and $r\geq 48$, then $n_1=2$ or $3$ and $|c_1|\leq n_1$.
\end{lemma}
\begin{proof}
    Assume we have a configuration of rank $r$ with dimension $\Delta\leq\frac{r}{3}-12$ with column widths $q_i$. Write $r=n_iq_i+c_i$ with $\frac{-q_i}{2}< c_i\leq \frac{q_i}{2}$. Suppose $n_1\geq 4$. Then $3q_1\geq r=n_1q_1+c_1\geq 4q_1+c_1\geq \frac{7}{2}q_1$
    which is not possible. Suppose $n_1=1$. Then $r=q_1+c_1\leq \frac{3q_1}{2}.$ Observe that $c_1\geq 2$ since $q_1+q_2+q_3=r$ for $q_2,q_3\geq 1$. Then
    $$\beta_1\geq c_1(q-c_1)\geq c_1\frac{q_1}{2}\geq c_1\frac{r}{3}.$$
    If $c_1\geq 4$, then $\beta_1\geq \frac{4}{3}r$ and we would be done. If $c_1=3$, then $q_1=r-3$ and
    $$\beta_1\geq 3(q_1-3)=3r-18$$
    which is bigger than $\frac{4r}{3}-12$ once $r\geq 4$. If $c_1=2$, then $r=q_1+2$ and
    $$\beta_1\geq 2(q_1-2)=2(r-4)=2r-8.$$
    Hence, $n_1\in \{2,3\}$. 

    Now we show that $|c_1|\leq n_1$. Suppose $|c_1|>n_1$. Then using $q_1\geq \frac{r}{3}$,
    $$\beta_1\geq |c_1|(q_1-|c_1|)\geq |c_1|\frac{r}{6}.$$
    Once $|c_1|\geq 9$, this is bigger than $\frac{4r}{3}$ and hence $\Delta> \frac{r}{3}-12$. 

    Consider $|c_1|\leq 8$. First note that by Lemma \ref{lem_negcol} since $r-q_1\leq \frac{2}{3}r$, $\Delta_2+\Delta_3\geq -\frac{r}{2}$. In particular, in order to show that $\Delta> \frac{r}{3}-12$, it is sufficient to show $\Delta_1>\frac{r}{3}-12+\frac{r}{2}$. Let $n_1=2$. Then by Proposition \ref{prop_quadmincoldim}, 
    $$
    \Delta_1\geq (|c_1|-1)\left(\frac{r}{2}-\frac{|c_1|}{2}-|c_1|\right).$$
    Running through the possibilities $|c_1|\in \{3,4,5,6,7,8\}$, $\Delta_1> \frac{r}{3}-12+\frac{r}{2}$ once $r\geq 0,9,16,20,24,28$ respectively. Hence, when $n_1=2$, $|c_1|\leq 2$.

    Now suppose $n_1=3$ and assume $|c_1|>3$. Then by Proposition \ref{prop_mincoldim},
    $$\Delta_1\geq |c_1|(q_1-|c_1|)-|c_1|\geq |c_1|\left(\frac{r}{3}-\frac{|c_1|}{3}-|c_1|\right)-|c_1|.$$
    Running though the possibilities $|c_1|\in \{4,5,6,7,8\}$, $\Delta_1>\frac{r}{3}-12+\frac{r}{2}$ once $r\geq 25, 30,34, 38, 42$. Hence when $n_1=3$, we have $|c_1|\leq 3$.
\end{proof}

\begin{lemma}\label{lem_q2c2res}
    If $\Delta\leq \frac{r}{3}-12$ and $r\geq 48$, then $q_2\geq \frac{2r}{9}$, $n_2=2,3,4$ and $|c_2|\leq n_2$.
\end{lemma}
\begin{proof}
    Assume we have a configuration of rank $r$ with dimension $\Delta\leq\frac{r}{3}-12$ with column widths $q_i$. Write $r=n_iq_i+c_i$ with $\frac{-q_i}{2}< c_i\leq \frac{q_i}{2}$.
    By Lemma \ref{lem_q1c1res}, $q_1\leq \frac{r}{2}+1$. Then $q_2\geq \frac{1}{2}(r-q_1)\geq \frac{r}{4}-\frac{1}{2}.$ Once $r>18$, $q_2> \frac{2r}{9}$. 
    We now proceed as we did for Lemma \ref{lem_q1c1res}. Suppose $n_2\geq 5$. Then
    $$\frac{9}{2}q_2> r=n_2q_2+c_2\geq 5q_2-\frac{q_2}{2}=\frac{9}{2}q_2$$
    which is a contradiction. 
    
    Suppose $n_2=1$. We know that
    $$\frac{r}{2}+1\geq q_1\geq q_2$$
    so $r\geq 2q_2+2$. But if $n_2=1$, then $r=q_2+c_2\leq \frac{3}{2}q_2$ which is a contradiction. Thus, $n_2\in \{2,3,4\}$.

    We aim to show that $|c_2|\leq n_2$. Assume $|c_2|>n_2$. We see that
    $$\beta_2\geq |c_2|(q_2-|c_2|)\geq |c_2|\frac{q_2}{2}\geq |c_2|\frac{r}{9}.$$
    Once $|c_2|\geq 13$, for all $r$, we have $\beta_2> \frac{4r}{3}-12$.

    In order to show $\Delta>\frac{r}{3}-12$ when $|c_2|>n_2$, it is sufficient to show that $\Delta_2>\frac{r}{3}-12+\frac{r}{2}+\frac{r}{8}$. We can see this as follows. Suppose first that $\Delta_1<0$, then by Lemma \ref{lem_q1c1res} and \ref{lem_negcol}, $q_1=\frac{r}{2}$ and $\Delta_1=\frac{-r}{2}$. If $\Delta_3<0$ as well, then $q_3=\frac{r}{n_3}$ with $n_3\in \{4,6,8,10,...\}$. However, if $n_3=4$ or $n_3=6$, then $q_2=\frac{r}{4}$ or $q_2=\frac{r}{3}$ respectively and $c_2=0<n_2$. Hence $\Delta_3\geq \frac{-r}{8}$. If we suppose instead that $\Delta_1\geq 0$, then $\Delta_1+\Delta_3\geq \Delta_3\geq \frac{-r}{4}>-\frac{r}{2}-\frac{r}{8}$.
    
    Let $n_2=2$ and assume $|c_2|>2$. By Proposition \ref{prop_quadmincoldim}, 
    $$\Delta_2\geq |c_2|(q_2-|c_2|)-q_2+|c_2|\geq (|c_2|-1)\left(\frac{r}{2}-\frac{|c_2|}{2}-|c_2|\right).$$
    Running through the possibilities $|c_2|\in \{3,4,5,6,7,8,9,10,11,12\}$, this is bigger than $\frac{r}{3}-12+\frac{r}{2}+\frac{r}{8}$ once $r\geq 0,12, 18, 22, 25, 29, 32, 35, 38, 41$ respectively. 

    Let $n_2=3$. Since $q_2\leq \frac{r}{3}$, we have $c_2\geq 0$. Assume $c_2>3$. By Proposition \ref{prop_quadmincoldim}, 
    $$\Delta_2\geq c_2(q_2-c_2)-c_2= c_2\left(\frac{r}{3}-\frac{c_2}{3}-c_2\right)-c_2.$$
    Running through the possibilities $c_2\in \{4,5,6,7,8,9,10,11,12\}$, $\Delta_2>\frac{r}{3}-12+\frac{r}{2}+\frac{r}{8}$ once $r\geq 36,38,41,44,48,52,56,60,64$. We could potentially have issues with $c_2\in \{8,9,10,11,12\}$, however note that in order to have $c_2$, we must have $q_2\geq 2|c_2|$ and $r\geq 6|c_2|+|c_2|$. Then it is only possible to have $c_2=8,9,10,11,12$ once $r\geq 56,63,70,77,84$. Comparing these to the inequalities achieved above, we can conclude that for all $r\geq 44$, if $c_2>3$, then $\Delta>\frac{r}{3}-12$.

    Let $n_2=4$ and suppose $|c_2|>4$. We see that $\Delta_1+\Delta_3\geq -\frac{r}{2}$ as follows. First suppose $\Delta_1<0$ so $q_1=\frac{r}{2}$. Then $q_3=\frac{r}{4}+\frac{c_2}{4}>\frac{r}{4}-3$. Note that $\frac{r}{4}-3>\frac{r}{6}$ once $r\geq 37$. Then if $\Delta_3<0$, $q_3=\frac{r}{4}$, but this would imply $c_2=0$. Now suppose instead that $q_1\neq \frac{r}{2}$. Then $\Delta_1\geq 0$ and  $\Delta_3\geq -\frac{r}{4}$. Either way, $\Delta_1+\Delta_3\geq \frac{-r}{2}$.
    
    By Proposition \ref{prop_quadmincoldim},
    $$
    \Delta_2\geq (|c_2|-1)\left(\frac{r}{4}-\frac{c_2}{4}-|c_2|\right).
    $$
    Suppose $c_2<0$. Then for $c_2\in \{-5,-6,-7,-8,-9,-10,-11,-12\}$, the above gives $\Delta_2>\frac{r}{3}-12+\frac{r}{2}$ once $r\geq 19,26,30,33,37,40,43,46$.

    Now suppose $c_2>0$. We cannot have $q_1=\frac{r}{2}$ since then $q_2=\frac{r}{4}-\frac{c_2}{4}<\frac{r}{4}+\frac{c_2}{4}=q_3$. It is sufficient then to show $\Delta_2>\frac{r}{3}-12+\frac{r}{4}$. The above gives this for $c_2\in \{5,6,7,8,9,10,11,12\}$ once $r\geq 32, 39, 45, 50,56, 61, 66, 71$. Configurations with $c_2=8,9,10,11,12$ occur only once $r\geq 4(2c_2)+c_2=72,81,90,99,108$. Hence, we can conclude that $|c_2|\leq 4=n_2$ for all $r\geq 48$ as desired.
\end{proof}

\begin{prop}\label{prop_nnpair}
    If $\Delta\leq \frac{r}{3}-12, r\geq 48$, then $(n_1,n_2)\in \{(2,2),(3,3),(2,4),(2,3)\}$ and $|c_i|\leq n_i$.
\end{prop}
\begin{proof}
    From Lemmas \ref{lem_q1c1res} and \ref{lem_q2c2res}, $n_1$ is either $2$ or $3$ and $n_2$ is one of $2,3,4$ and that $|c_i|\leq n_i$. It is left to show that if $n_1=3$, then $n_2=3$.

    Assume $n_1=3$. Then $q_1=\frac{r}{3}-\frac{c_1}{3}$ with $-3\leq c_1\leq 0$.
    If $n_2=2$, then $q_2=\frac{r}{2}+\frac{c_2}{3}$ with $|c_2|\leq 2$. Since $q_1\geq q_2$, then
    $$\frac{r}{3}+1\geq q_1\geq q_2\geq \frac{r}{2}-1$$
    which cannot happen once $r>12$. 

    If $n_2=4$, then $q_2\leq \frac{r}{4}+1$. By the order assumption, $q_3\geq \frac{1}{2}(r-q_1)\geq \frac{r}{3}-1$. But then
    $$\frac{r}{4}+1\geq \frac{r}{3}-1$$
    which cannot happen once $r>24$. Hence if $n_1=3$, then $n_2=3$.
\end{proof}

\section{$(n_1,n_2)$ is $(2,2), (3,3), (2,4),(2,3)$}\label{sec_5}

Recall that $n_i,c_i$ are defined by $r=n_iq_i+c_i$ with $\frac{-q_i}{2}< c_i\leq \frac{q_i}{2}$. In Prop \ref{prop_nnpair}, we reduced the non-overlapping configurations to those with constraints $(n_1,n_2)=(3,3)$ or $(2,4)$ or $(2,3)$ and $|c_i|\leq n_i, i=1,2$. In this section, we work through these possibilities. We will find our minimal cases when $(n_1,n_2)=(2,3)$, $c_1=0$, $c_2=0$. Let's describe our general approach.

\begin{steps}
    \item Knowing $(n_1,c_1)$ and $(n_2,c_2)$ determines $q_1,q_2, q_3$ in terms of $r$. Use Prop \ref{prop_quadmincoldim} to get the minimal (possibly negative or zero) dimension of some or all of the $\Delta_i$ to get a lower bound on the dimension. If this is positive and sufficiently big for $r\geq 48$, then we stop and compare it to $\frac{r}{3}-12$.
    \item Let $\Delta_1^0, \Delta_2^0, \Delta_3^0$ denote the minimal column dimensions for the first, second, and third columns, respectively. Recall from before Lemma \ref{lem_q1c1res} that this is
    $$\Delta_i^0=\begin{cases}
        |c_i|(q_i-|c_i|)-q_i+|c_i| & \text{if }n_i \text{ even}\\
        |c_i|(q_i-|c_i|)-|c_i| & \text{if }n_i \text{ odd}.
    \end{cases}$$
    We then go column by column, changing the column partition.
    \item Apply moves iteratively to the partition $(q_1,...,q_1,k_1)$. Given a partition, use Lemma \ref{lem_colqq_qqa} to see if $\Delta_1+\Delta_2^0+\Delta_3^0>0$. If it is, then we can compare this to $\frac{r}{3}-12$ and rule out all partitions less than that one in the dominance ordering by Prop \ref{prop_quadDelta_domorder}. If it is not, apply another move to the partition and repeat. We do this until we have a list of all possible column dimensions for $\Delta_1$ such that $\Delta_1+\Delta_2^0+\Delta_3^0\leq 0$. 
    \item Do Step 3 for the second and third column to collect all the possible column dimensions $\Delta_i$ such that $\Delta_1^0+\Delta_2+\Delta_3^0\leq 0$ and $\Delta_1^0+\Delta_2^0+\Delta_3\leq 0$ respectively. 
    \item We will then have a small list of column dimensions that have not been ruled out for the first, second, and third columns. We will always find that for all of the column dimensions on this list that $\Delta_1+\Delta_2+\Delta_3$ will be either $\leq 0$ or $\geq \frac{r}{3}-12$.
\end{steps}

\subsection{$(n_1,n_2)=(2,2)$}

When $(n_1,n_2)=(2,2)$, we have
$$r=2q_1+c_1, q_1=\frac{r}{2}-\frac{c_1}{2}, -2\leq c_1\leq 2$$
$$r=2q_2+c_2, q_2=\frac{r}{2}-\frac{c_2}{2}, -2\leq c_2\leq 2$$
and $q_3=\frac{c_1}{2}+\frac{c_2}{2}$. 

\subsubsection*{Case: $|c_1|=2$}

If $c_1=2$, then $r=2q_1+2, q_1=\frac{r}{2}-1$. Since $q_2\leq q_1$, we have $q_1=q_2$ and $q_3=2$. It follows by Prop \ref{prop_quadmincoldim} that
$$\Delta_i^0\geq q_i-2=\frac{r}{2}-3, i=1,2.$$
Then $\Delta\geq \Delta_1^0+\Delta_2^0-2= r-8$. If $c_1=-2$, then $q_1=\frac{r}{2}+1$ and so $q_2\leq \frac{r}{2}-2$ since $q_3\geq 1$. This contradicts $|c_2|\leq 2$.

\subsubsection*{Case: $|c_1|=1$}

If $c_1=1$, then $r=2q_1+1, q_1=\frac{r}{2}-\frac{1}{2}$. Since $q_2\leq q_1$ and $r$ is odd, this implies $q_1=q_2$ and $q_3=1$. By Prop \ref{prop_quadmincoldim}, the minimal column dimension of each column is $0$. We look at the next possible choices of configurations. For the third column, we always have $\Delta_3=0$. Since $q_1=q_2$, without loss of generality, suppose we change the first column. If we use $(q_1,q_1,1)$, then $\Delta_1\in \{0,r-3\}$. Any other partition is $(q_1,q_1-1,2)$ or lower. If we use $(q_1,q_1-1,2)$, by Lemma \ref{lem_colqq_qqa}, $\Delta_1\geq 2q_1-2=\frac{r}{2}-3$. Hence, the smallest nonzero dimension with $c_1=1$ is 
$$\Delta=\frac{r}{2}-3>\frac{r}{3}-12.$$
If $c_1=-1$, then $2q_1-1, q_1=\frac{r}{2}+\frac{1}{2}$. But then $q_2\leq r-q_1-1=\frac{r}{2}-\frac{3}{2}$ which contradicts $|c_2|\leq n_2=2$.

\subsubsection*{Case: $c_1=0$}

If $c_1=0$, then $r=2q_1$ and $q_1=\frac{r}{2}$. Since $q_2\leq r-q_1-q_2\leq \frac{r}{2}-1$, we must have $c_2=2$ and $q_3=1$. By Prop \ref{prop_quadmincoldim}, the minimal dimensions for $\Delta_i$ are $\Delta_1^0=\frac{-r}{2}, \Delta_2^0=\frac{r}{2}-3$, and $\Delta_3^0=-1$. Taking these gives $\Delta=-4$. Note that $\Delta_3$ can only be $1$ or $-1$, so unless we change the other columns, the dimension would be negative. If $\Delta_1\neq \frac{-r}{2}$, then $\Delta_1\geq 0$  and $\Delta\geq \frac{r}{2}-4$. Hence, we can assume $\Delta_1=\frac{-r}{2}$. 

Let's look at changing the second column. Below, we record all dimensions that certain partitions can achieve using Lemma \ref{lem_colqq_qqa}.
\begin{enumerate}
    \item $(q_2,q_2,2)$, $\Delta_2\in \left\{\frac{r}{2}-3, \frac{3r}{2}-9\right\}$
    \item $(q_2,q_2,1,1)$, $\Delta_2\in \left\{\frac{r}{2}-3, \frac{r}{2}-1,\frac{3r}{2}-7,\frac{3r}{2}-5\right\}$
    \item $(q_2,q_2-1,3)$, $\Delta_2\in \left\{\frac{3r}{2}-9, \frac{5r}{2}-19 \right\}$
\end{enumerate}

Observe that any other possible column partition of $r$ of width $q_2$ are less in the dominance ordering than (iii) $(q_2,q_2-1,3)$ and hence by Prop \ref{prop_quadDelta_domorder}, satisfy $\Delta\geq \mathrm{min}\left\{\frac{5r}{2}-19, \frac{3r}{2}-9\right\}$. For $r\geq 11$, $\frac{5r}{2}-19>\frac{3r}{2}-9$. It follows that if $\Delta_2$ is not $\frac{r}{2}-3$ or $\frac{r}{2}-1$, then
$$\Delta\geq \frac{-r}{2}+\left(\frac{3r}{2}-9\right)-1=r-10.$$
If we use $\Delta_2=\frac{r}{2}-1$, then
$$\Delta\leq \frac{-r}{2}+\left(\frac{r}{2}-1\right)+1=0.$$
Thus any positive dimension with $c_1=0$ satisfies $\Delta\geq r-10>\frac{r}{3}-12$. 

\subsection{$(n_1,n_2)=(3,3)$}

When $(n_1,n_2)=(3,3)$, we have
$$r=3q_1+c_1, q_1=\frac{r}{3}-\frac{c_1}{3}, -3\leq c_1\leq 0,$$
$$r=3q_2+c_2, q_2=\frac{r}{3}-\frac{c_2}{3}, -3\leq c_2\leq 3,$$
and $q_3=\frac{r}{3}+\frac{c_1+c_2}{3}$. Note that $\frac{r}{3}+1\geq q_i\geq \frac{r}{3}-1$ for $i=1,2$ and $\frac{r}{3}+2\geq q_3\geq \frac{r}{3}-2$. By Lemma \ref{lem_negcol}, once $r\geq 25$, $\Delta_i\geq 0$ for each $i$ so $\Delta\geq \Delta_i$.

\subsubsection*{Case: $c_1=-3, -2,-1$}

If $c_1=-3$, then $q_1=\frac{r}{3}+1$ and $\Delta\geq \Delta_1^0\geq r-9.$ 

\noindent If $c_1=-2$, then $q_1=\frac{r}{3}+\frac{2}{3}$ and $\Delta\geq \Delta_1^0\geq \frac{2r}{3}-\frac{14}{3}.$ 

\noindent If $c_1=-1$, then $q_1=\frac{r}{3}+\frac{1}{3}$ and $\Delta\geq \Delta_1^0\geq \frac{r}{3}-\frac{5}{3}.$

\subsubsection*{Case: $c_1=0$}

When $c_1=0$, then $q_1=q_2=q_3=\frac{r}{3}$. The minimal dimension is $0$. We then look at changing one of the columns. Without loss of generality, suppose we change the first column. If we use the partition $(q_1,q_1,q_1)$, $\Delta_1=0$. If we use $(q_1,q_1,q_1-1,1)$, $\Delta_1\in \left\{\frac{2r}{3}-2,\frac{2r}{3}-4,\frac{2r}{3}\right\}$. Any other partition is less than $(q_1,q_1,q_1-1,1)$ and so
$$\Delta\geq \frac{2r}{3}-4.$$

\subsection{$(n_1,n_2)=(2,4)$}

When $(n_1,n_2)=(2,4)$,
$$r=2q_1+c_1, q_1=\frac{r}{2}-\frac{c_1}{2}, |c_1|\leq 2,$$
$$r=4q_2+c_2, q_2=\frac{r}{4}-\frac{c_2}{4}, |c_2|\leq 4$$
and $q_3=\frac{r}{4}+\frac{2c_1+c_2}{4}$. Note that since $q_1\geq \frac{r}{2}-1$, and for $i=2,3$, $\frac{r}{4}+2\geq q_i\geq \frac{r}{4}-2$. By Lemma \ref{lem_negcol} once $r\geq 25$, $\Delta_1<0$ only if $q_1=\frac{r}{2}$, and for $i=2,3$ $\Delta_i<0$ only if $q_i=\frac{r}{4}$. 

\subsubsection*{Case: $|c_1|=2$}

When $|c_1|=2$, $r=2q_1\pm 2$ and $q_1=\frac{r}{2}\mp 1$. Then
$$\Delta_1^0\geq \frac{r}{2}\mp 1-2.$$
If $\Delta_2+\Delta_3\geq 0$, we are done. Suppose one of $q_2,q_3$ are $\frac{r}{4}$, then the other is $\frac{r}{4}\pm 1$. It follows that
$$\Delta_2^0+\Delta_3^0\geq \frac{-r}{4}+\left(\frac{3r}{4}-15\right)=\frac{r}{2}-15$$
which is strictly greater than zero once $r\geq 31$. 

\subsubsection*{Case: $|c_1|=1$}

The minimal first column dimension is $\Delta_1^0=0$. Since $r$ is odd, $\Delta_i\geq 0$ for each $i$ and $c_2\in \{-3,-1,1,3\}$. 

\noindent\textbf{Subcase:} $|c_2|=3$. If $|c_2|=3$, $r=4q_3\pm 3$ and $q_3=\frac{r}{4}\mp \frac{3}{4}$. Then
$$\Delta\geq \Delta_2^0\geq \frac{3r}{4}-\frac{57}{6}$$
which once $r\geq 6$ is greater than $\frac{r}{3}-12$. 

\noindent \textbf{Subcase:} $c_2=1$. If $c_2=1$, then $r=4q_2+1$ and $q_2=\frac{r}{4}-\frac{1}{4}$. Since
$$q_3=\frac{r}{4}+\frac{2c_1+1}{4}\leq q_2,$$
we must have $c_1=-1$ and $q_2=q_3$. The minimal second and third column dimensions are $\Delta_2^0=\Delta_3^0=0$. 

Let's first look at changing the first column. Consider the following possible column partitions listed with all their possible dimensions.
\begin{enumerate}
    \item $(q_1,q_1-1)$, $\Delta_1\in \{0,r-1\}$
    \item $(q_1-1,q_1-1,1)$, $\Delta_1\in \{r+5, r+7, 2r+2,2r+4\}$
    \item $(q_1,q_1-2,1)$, $\Delta_1\in \{r-1, 2r-6\}$
\end{enumerate}
Any other partition is less than (iii) $(q_1,q_1-2,1)$ and hence satisfies $\Delta_1\geq r-1$ once $r\geq 6$. It follows that unless we have $\Delta_1=0$, $\Delta\geq r-1$. Now we can assume $\Delta_1=0$.

Now let's look at changing the second column. Consider the following partitions and their possible dimensions.
\begin{enumerate}
    \item $(q_2,q_2,q_2,q_2,1)$, $\Delta_2\in \left\{0,\frac{r}{2}-\frac{5}{2}\right\}$
    \item $(q_2,q_2,q_2,q_2-1,2)$, $\Delta_2\in \left\{\frac{r}{2}-\frac{1}{2}, r-7\right\}$
\end{enumerate}

Any other partition with all parts less than or equal to $q_2$ will be less than (ii) $(q_2,q_2,q_2,q_2-1,2)$ and hence satisfy $\Delta_2\geq \frac{r}{2}-\frac{1}{2}$ once $r\geq 14$. It follows that unless we take $\Delta_2=0$, then $\Delta\geq \frac{r}{2}-\frac{1}{2}$. 

Since $q_2=q_3$, we are left with only the possibility of $\Delta_1=\Delta_2=\Delta_3=0$, but this, of course, gives zero dimension. 

\noindent \textbf{Subcase:} $c_2=-1$. When $c_2=-1$, $r=4q_1-1$, $q_2=\frac{r}{4}+\frac{1}{4}$, and $q_3=\frac{r}{4}+\frac{2c_1-1}{4}$. 

First suppose $c_1=-1$, then $q_3=\frac{r}{4}-\frac{3}{4}$ and $r=4q_3+3$. Once $r\geq 19$, $q_3\geq 4$ and
$$\Delta\geq \Delta_3^0\geq \frac{r}{2}-\frac{15}{2}.$$
Next, consider $c_1=1$. This gives $q_2=q_3$. The minimal column dimensions are $\Delta_1^0=\Delta_2^0=\Delta_3^0=0$. Let's look at changing the first column. See that,
\begin{enumerate}
    \item $(q_1,q_1-1)$, $\Delta_1\in \{0,r-3\}$
    \item $(q_2,q_2-1,1)$, $\Delta_1\in \left\{\frac{r}{2}-\frac{1}{3}, \frac{3r}{2}-\frac{11}{2}\right\}$.
\end{enumerate}

Any other partition is less than (ii)$(q_2,q_2-1,1)$ and so satisfies $\Delta_2\geq \frac{r}{2}-\frac{1}{3}$ once $r\geq 6$. It follows that unless $\Delta_1=0$, $\Delta\geq \frac{r}{2}-\frac{1}{3}.$ We can now assume $\Delta_1=0$. 

Let's look at changing the second column.
\begin{enumerate}
    \item $(q_2,q_2,q_2,q_2-1)$, $\Delta_1\in \left\{0,\frac{r}{2}-\frac{3}{2}\right\}$
    \item $(q_2,q_2,q_2-1,q_2-1,1)$, $\Delta_1\in \left\{\frac{r}{2}-\frac{3}{2}, \frac{r}{2}+\frac{1}{2}, r-5,r-3\right\}$
    \item $(q_2,q_2,q_2,q_2-2,1)$, $\Delta_2\in \left\{\frac{r}{2}-\frac{3}{2}, r-7\right\}$.
\end{enumerate}

Any other partition is less than one of (ii) $(q_2,q_2,q_2-1,q_2-1,1)$ or (iii) $(q_2,q_2,q_2,q_2-2,1)$ and hence satisfies $\Delta_2\geq \frac{r}{2}-\frac{3}{2}$ once $r\geq 12$. In particular, unless $\Delta_2=0$, $\Delta\geq \frac{r}{2}-\frac{3}{2}$. Hence, we can assume $\Delta_2=0$. However, $q_2=q_3$, so the same applies to the third column and $\Delta_3=0$. 

\subsubsection*{Case: $c_1=0$}

When $c_1=0$, $r=2q_1$ and $q_1=\frac{r}{2}$. The minimal dimension of the first column is $\Delta_1^0=\frac{-r}{2}$. Note that $q_2\geq \frac{r}{4}$ and $r$ is even, so $c_2\in \{-4,-2,0\}$. Once $r\geq 36$, $q_3\geq 8$ and $c_3=-c_2$.

\noindent \textbf{Subcase:} $c_2=-4$. In this subcase,
$$r=4q_2-4,q_2=\frac{r}{4}+1,$$
$$r=4q_3+4,q_3=\frac{r}{4}-1.$$
Then
$$\Delta_2+\Delta_3\geq \left(\frac{3r}{4}-9\right)+\left(\frac{3r}{4}-15\right)=\frac{3r}{2}-24.$$
It follows that $\Delta\geq r-24$ which is greater than $\frac{r}{3}-12$ once $r\geq 19$. 

\noindent \textbf{Subcase:} $c_2=-2$. In this subcase,
$$r=4q_2-4, q_2=\frac{r}{4}+\frac{1}{2},$$
$$r=4q_2+2, q_3=\frac{r}{4}-\frac{1}{2}.$$
The minimal column dimensions are $\Delta_2^0=\frac{r}{4}-\frac{3}{2}$ and $\Delta_3^0=\frac{r}{4}-\frac{5}{2}$. The minimal dimension is then $\Delta=-4$. If $\Delta_1\neq \frac{-r}{2}$, then $\Delta_1\geq 0$ and $\Delta\geq \frac{r}{2}-4$. We can now assume $\Delta_1=\frac{-r}{2}$. 

Let's look at changing the second column. 
\begin{enumerate}
    \item $(q_2,q_2,q_2,q_2-2)$, $\Delta_2\in \left\{r-4,\frac{r}{4}-\frac{3}{2}\right\}$
    \item $(q_2,q_2,q_2-1,q_2-1)$, $\Delta_2\in \left\{\frac{3r}{4}-\frac{3}{2}, \frac{r}{4}-\frac{3}{2}, \frac{3r}{4}-\frac{5}{2}, \frac{r}{4}+\frac{1}{2}\right\}$
    \item $(q_2,q_2,q_2-1,q_2-2,1)$, $\Delta_2\in \left\{\frac{5r}{4}-\frac{11}{2}, \frac{3r}{4}-\frac{5}{2}, \frac{5r}{4}-\frac{15}{2}, \frac{3r}{4}-\frac{1}{2}\right\}$
    \item $(q_2,q_2,q_2,q_2-3,1)$, $\Delta_2\in \left\{\frac{5r}{4}-\frac{23}{2}, \frac{3r}{4}-\frac{9}{2}\right\}$
\end{enumerate}

Any other partition will be less than (iii) $(q_2,q_2,q_2-1,q_2-2,1)$ or (iv) $(q_2,q_2,q_2,q_2-3,1)$ and hence satisfies $\Delta_2\geq \frac{3r}{4}-\frac{9}{2}$ once $r\geq 15$. Unless we have $\Delta_2\in \left\{\frac{r}{4}-\frac{3}{2}, \frac{r}{4}+\frac{1}{2}\right\}$,
$$\Delta\geq \frac{-r}{2}+\left(\frac{3r}{4}-\frac{9}{2}\right)+\left(\frac{r}{4}-\frac{5}{2}\right)=\frac{r}{2}-7.$$
Thus we can assume $\Delta_2\in \left\{\frac{r}{4}-\frac{3}{2}, \frac{r}{4}+\frac{1}{2}\right\}$.

Let's look at changing the third column.
\begin{enumerate}
    \item $(q_3,q_3,q_3,q_3,2)$, $\Delta_3\in \left\{\frac{r}{4}-\frac{5}{2}, \frac{3r}{4}-\frac{15}{2}\right\}$
    \item $(q_3,q_3,q_3,q_3,1,1)$, $\Delta_3\in \left\{\frac{r}{4}-\frac{5}{2}, \frac{r}{4}-\frac{1}{2}, \frac{3r}{4}-\frac{11}{2}, \frac{3r}{4}-\frac{7}{2}\right\}$
    \item $(q_3,q_3,q_3,q_3-1,3)$, $\Delta_3\in \left\{\frac{3r}{4}-\frac{15}{2}, \frac{5r}{4}-\frac{33}{2}\right\}$
\end{enumerate}
Observe that any other partition will be less than $(q_3,q_3,q_3,q_3-1,3)$ and hence satisfies $\Delta_3\geq \frac{3r}{4}-\frac{15}{2}$ once $r\geq 19$. If $\Delta_3\notin \left\{\frac{r}{4}-\frac{5}{2}, \frac{r}{4}-\frac{1}{2}\right\}$, then
$$\Delta\geq \frac{-r}{2}+\left(\frac{r}{4}-\frac{3}{2}\right)+\left(\frac{3r}{4}-\frac{15}{2}\right)=\frac{r}{2}-9.$$
We can then assume $\Delta_3\in \left\{\frac{r}{4}-\frac{5}{2}, \frac{r}{4}-\frac{1}{2}\right\}$. We have reduced the column dimensions to
$$\Delta_1=\frac{-r}{2}, \Delta_2\in \left\{\frac{r}{4}-\frac{3}{2}, \frac{r}{4}+\frac{1}{2}\right\},\Delta_3\in \left\{\frac{r}{4}-\frac{5}{2}, \frac{r}{4}-\frac{1}{2}\right\}.$$
However, observe that taking any combination of these $\Delta_1+\Delta_2+\Delta_3$ will never be positive. 

\noindent \textbf{Subcase:} $c_2=0$. In this subcase, $q_1=\frac{r}{2}, q_2=q_3=\frac{r}{4}$ and the minimal column dimensions are $\Delta_1^0=\frac{-r}{2}$ and $\Delta_2^0=\Delta_3^0=\frac{-r}{4}$. Let's first consider changing the first column.
\begin{enumerate}
    \item $(q_1,q_1)$, $\Delta_1\in \left\{\frac{-r}{2}, \frac{r}{2}\right\}$
    \item $(q_1,q_1-1,1)$, $\Delta_1\in \left\{\frac{3r}{2}-4, \frac{5r}{2}-12\right\}$
    \item $(q_1-1,q_1-1,2)$, $\Delta_1\in \left\{\frac{3r}{2}-4,\frac{3r}{2}-2,\frac{5r}{2}-10,\frac{5r}{2}-8\right\}$
\end{enumerate}
Notice that any other partition will be less than (ii) $(q_1,q_1-1,1)$ or (iii) $(q_1-1,q_1-1,2)$ and hence satisfies $\Delta_1\geq \frac{3r}{2}-4$ once $r\geq 9$. Unless $\Delta_1\in \left\{\frac{r}{2}, \frac{-r}{2}\right\}$, 
$$\Delta\geq \frac{3r}{2}-4-\frac{r}{2}=r-4.$$
Now we assume $\Delta_1\in \left\{\frac{-r}{2}, \frac{r}{2}\right\}$. Let's consider changing the second column. If we use $(q_2,q_2,q_2,q_2)$, then $\Delta_2\in \left\{\frac{-r}{4},\frac{r}{4}\right\}$. Only one partition is achieved by applying one move to the Young diagram. It is $(q_2,q_2,q_2,q_2-1,1)$ which has possible dimensions $\Delta_2\in \left\{\frac{r}{4}, \frac{3r}{4}-6\right\}$. Next, consider partitions achieved by applying two moves to the Young diagram,
\begin{enumerate}
    \item $(q_2,q_2,q_2-1,q_2-1,2)$, $\Delta_2\in \left\{\frac{3r}{4}-4, \frac{3r}{4}-2, \frac{5r}{4}-10,\frac{5r}{4}-8\right\}$
    \item $(q_2,q_2,q_2-1,q_2-1,1,1)$, $\Delta_2\in \left\{\frac{3r}{4}-4, \frac{3r}{4}-2, \frac{3r}{4}, \frac{5r}{4}-8, \frac{5r}{4}-6, \frac{5r}{4}-4\right\}$
    \item $(q_2,q_2,q_2,q_2-2,2)$, $\Delta_2\in \left\{\frac{3r}{4}-4, \frac{5r}{4}-12\right\}$
    \item $(q_2,q_2,q_2,q_2-2,1,1)$, $\Delta_2\in \left\{\frac{r}{2}-4,\frac{r}{2}-2, \frac{5r}{4}-10,\frac{5r}{4}-8\right\}$
\end{enumerate}

Any other partition is lower than one of the following two partitions:
\begin{enumerate}
    \item $(q_2,q_2-1,q_2-1,q_2-1,3)$, $\Delta_2\in \left\{\frac{5r}{4}-6, \frac{7r}{4}-14\right\}$
    \item $(q_2,q_2,q_2,q_2-3,3)$, $\Delta_2\in \left\{\frac{5r}{4}-12, \frac{7r}{4}-24\right\}$
\end{enumerate}

Any other partition is less than one of these two and so satisfies $\Delta_2\geq \frac{5r}{4}-12$ once $r\geq 25$. If $\Delta_2\notin \left\{\frac{-r}{4}, \frac{r}{4},\frac{r}{2}-4,\frac{r}{2}-2, \frac{r}{2},\frac{r}{2}+2, \frac{3r}{4}-6,\frac{3r}{4}-4,\frac{3r}{4}-2, \frac{3r}{4}\right\}$, then
$$\Delta\geq \frac{-r}{2}+\left(\frac{5r}{4}-12\right)+\frac{-r}{4}=\frac{r}{2}-12.$$
Hence, we can assume $\Delta_2$ is one of those values. Since $q_2=q_3$, this also holds for the third column. We have then reduced the possible column dimensions to
$$\Delta_1\in \left\{\frac{-r}{2}, \frac{r}{2}\right\},$$
$$\Delta_i\in \left\{\frac{-r}{4}, \frac{r}{4},\frac{r}{2}-4,\frac{r}{2}-2, \frac{r}{2},\frac{r}{2}+2, \frac{3r}{4}-6,\frac{3r}{4}-4,\frac{3r}{4}-2, \frac{3r}{4}\right\}.$$
Note that if $\Delta_1=\frac{r}{2}$ and at least one of $\Delta_i$, $i=2,3$ is not $\frac{-r}{4}$, then once $r\geq 17$,
$$\Delta\geq \frac{r}{2}+\frac{r}{4}-\frac{r}{4}=\frac{r}{2}.$$
Hence, we can assume $\Delta_1=\frac{-r}{2}$. If one of $\Delta_i$ is $\frac{-r}{4}$, then for $r\geq 3$,
$$\Delta\leq \frac{-r}{2}+\frac{-r}{4}+\frac{3r}{4}=0.$$
Hence, we can assume neither $\Delta_i$ for $i=2,3$ are $\frac{-r}{4}$. To get nonzero dimension, we need at least one of the $\Delta_i$ to not equal $\frac{r}{4}$ and so
$$\Delta\geq \frac{-r}{2}+\frac{r}{4}+\frac{3r}{4}-6=\frac{r}{2}-6.$$
Thus, whatever choices we make, if the dimension is positive, then $\Delta>\frac{r}{3}-12$.

\subsection{$(n_2,n_3)=(2,3)$}

Now we assume $(n_1,n_2)=(2,3)$. This means that we have 
$$r=2q_1+c_1, q_1=\frac{r}{2}-\frac{c_1}{2}, |c_1|\leq 2$$
$$r=3q_2+c_2, q_2=\frac{r}{3}-\frac{c_2}{3}, |c_2|\leq 4$$
and $q_3=\frac{r}{6}+\frac{3c_1+2c_2}{6}$. Note that once $r\geq 29$, we can assume $q_1>\frac{r}{4}$, $q_2>\frac{r}{4}$, and $\frac{r}{4}>q_3$. Then $\Delta_1<0$ only if $q_1=\frac{r}{2}$, $\Delta_2\geq 0$, and $\Delta_3\geq \frac{-r}{6}$.

\subsubsection*{Case: $|c_1|=2$}

In this case, $r=2q_1\pm 2$ and $q_1=\frac{r}{2}\mp 1$. Then
$$\Delta\geq \Delta_1^0+0-\frac
{r}{6}\geq \frac{r}{3}-3.$$

\subsubsection*{Case: $|c_1|=1$}
In this case, the minimal first column dimension is $\Delta_1^0=0$. Since $r$ is odd, $\Delta_i\geq 0$ for all $i$.

\noindent \textbf{Subcase:} $|c_2|=3$. In this subcase, $r=3q_2\pm 3$ and $q_2=\frac{r}{3}\mp 1$. Then
$$\Delta\geq \Delta_2^0= r-15$$
which is bigger than $\frac{r}{3}-12$ once $r\geq 5$. 

\noindent \textbf{Subcase:} $|c_2|=2$. In this subcase, $r=3q_2\pm 2$ and $q_2=\frac{r}{3}\mp \frac{2}{3}$. Then
$$\Delta\geq \Delta_2^0\geq \frac{2r}{3}-\frac{22}{2}.$$

\noindent \textbf{Subcase:} $|c_2|=1$. In this subcase, $r=3q_2\pm 1$ and $q_2=\frac{r}{3}\mp \frac{1}{3}$. Then
$$\Delta\geq \Delta_2^0\geq \frac{r}{3}-\frac{7}{3}.$$

\noindent \textbf{Subcase:} $c_2=0$. In this subcase, $r=3q_2$ and the minimal second column dimension is $\Delta_2=0$. If $c_1=1$, then $q_3=\frac{r}{6}-\frac{1}{2}$ and $r=6q_3+3$. If $c_1=-1$, then $q_3=\frac{r}{6}+\frac{1}{2}$ and $r=6q_3-3$. Once $r\geq 21$, $q_3\geq 4$ and we have $k_3$ is either $3$ and $q_3-3$ respectively. Then
$$\Delta\geq \Delta_3^0=\frac{r}{3}-7.$$

\subsubsection*{Case: $c_1=0$}

In this case, $q_1=\frac{r}{2}$, $r=2q_1$, and the minimal first column dimension is $\Delta_1=\frac{-r}{2}$. Note that $q_3=\frac{r}{6}$ only if $c_2=0$. 

\noindent \textbf{Subcase:} $|c_2|=3$. In this subcase, observe that $q_3\geq\frac{r}{6}-1$ which is greater than $\frac{r}{8}$ once $r\geq 25$ so $\Delta_3>0$. 
$$\Delta\geq \frac{-r}{2}+\Delta_2^0\geq \frac{-r}{2}+r-15=\frac{r}{2}-15$$
which is greater than $\frac{r}{3}-12$ once $r>19$. 

\noindent \textbf{Subcase:} $c_2=2$. In this subcase, $r=3q_2+2$, $q_2=\frac{r}{3}-\frac{2}{3}$ and $r=6q_3-4, q_3=\frac{r}{6}+\frac{2}{3}$. Once $r\geq 26$, $q_3\geq 4$ and $k_3$ is $q_3-4$. Then
$$\Delta\geq \frac{-r}{2}+\left(\frac{2r}{3}-\frac{22}{3}\right)+\left(\frac{r}{2}-10\right)=\frac{2r}{3}-\frac{52}{3}$$
which is greater than $\frac{r}{3}-12$ once $r\geq 17$. 

\noindent \textbf{Subcase:} $c_2=-2$. In this subcase, $r=3q_2-2, q_2=\frac{r}{3}+\frac{2}{3}$ and $r=6q_3+4, q_3=\frac{r}{6}-\frac{2}{3}$. Once $r\geq 34$, $q_3\geq 5$ and $k_3=4$. Then
$$\Delta\geq \frac{-r}{2}+\left(\frac{2r}{3}-\frac{14}{3}\right)+\left(\frac{r}{2}-14\right)=\frac{2r}{3}-\frac{56}{3}$$
which is greater than $\frac{r}{3}-12$ once $r\geq 21$. 

\noindent \textbf{Subcase:} $c_2=1$. In this subcase, $r=3q_2+1, q_2=\frac{r}{3}-\frac{1}{3}$ and $r=6q_3-2, q_3=\frac{r}{6}+\frac{1}{3}$. The minimal  dimensions are
$$\Delta_1^0=\frac{-r}{2}, \Delta_2^0=\frac{r}{3}-\frac{7}{3}, \Delta_3^0=\frac{r}{6}-\frac{5}{3}$$
and $\Delta^0=-4$. If $\Delta_1\neq \frac{-r}{2}$, then $\Delta_1\geq 0$ and $\Delta\geq \frac{r}{2}-4$. Hence, we assume $\Delta_1=\frac{-r}{2}$.

Let's look at changing the second column.
\begin{enumerate}
    \item $(q_2,q_2,q_2,1)$, $\Delta_2\in \left\{\frac{r}{3}-\frac{7}{3}, \frac{r}{3}-\frac{1}{3}\right\}$
    \item $(q_2,q_2,q_2-1,2)$, $\Delta_2\in \left\{r-9,r-7,r-5,r-3\right\}$
\end{enumerate}
Any other partition is less than (ii) $(q_2,q_2,q_2-1,2)$ and hence satisfies $\Delta_2\geq r-9$. It follows that if $\Delta_2\notin\left\{\frac{r}{3}-\frac{7}{3}, \frac{r}{3}-\frac{1}{3}\right\}$, then
$$\Delta\geq \frac{-r}{2}+\left(r-9\right)+\left(\frac{r}{6}-\frac{5}{3}\right)=\frac{2r}{3}-\frac{32}{3}.$$
We can then assume $\Delta_2\in\left\{\frac{r}{3}-\frac{7}{3}, \frac{r}{3}-\frac{1}{3}\right\}$. 

Let's look at changing the third column. 
\begin{enumerate}
    \item $(q_3,q_3,q_3,q_3,q_3,q_3-2)$, $\Delta_3\in \left\{\frac{r}{6}-\frac{5}{3}, \frac{r}{2}-7\right\}$
    \item $(q_3,q_3,q_3,q_3,q_3-1,q_3-1)$, $\Delta_3\in \left\{\frac{r}{6}-\frac{5}{3}, \frac{r}{6}+\frac{1}{3}, \frac{r}{2}-3, \frac{r}{2}-1\right\}$
    \item $(q_3,q_3,q_3,q_3,q_3-1,q_3-2,1)$, $\Delta_3\in \left\{\frac{r}{2}-3, \frac{r}{2}-1, \frac{5r}{6}-\frac{25}{3},\frac{5r}{6}-\frac{19}{3} \right\}$
    \item $(q_3,q_3,q_3,q_3,q_3,q_3-3,1)$, $\Delta_3\in \left\{\frac{r}{2}-5, \frac{5r}{6}-\frac{37}{3}\right\}$
    \item $(q_3,q_3,q_3,q_3-1,q_3-1,q_3-1,1)$, $\Delta_3\in \left\{\frac{r}{2}-1, \frac{5r}{6}-\frac{13}{3}\right\}$
\end{enumerate}

Any other partition is less than (iii) $(q_3,q_3,q_3,q_3,q_3-1,q_3-1)$, (iv) $(q_3,q_3,q_3,q_3,q_3,q_3-3,1)$, or (v) $(q_3,q_3,q_3,q_3-1,q_3-1,q_3-1,1)$ and hence satisfies $\Delta_3\geq \frac{r}{2}-7$ once $r\geq 17$. If $\Delta_3\notin \left\{\frac{r}{6}-\frac{5}{3}, \frac{r}{6}+\frac{1}{3}\right\}$, then
$$\Delta\geq \frac{-r}{2}+\left(\frac{r}{3}-\frac{7}{3}\right)+\left(\frac{r}{2}-7\right)=\frac{r}{3}-\frac{28}{3}.$$
We have now reduced our possible column dimensions to
$$\Delta_1=\frac{-r}{2}, \Delta_2\in\left\{\frac{r}{3}-\frac{7}{3}, \frac{r}{3}-\frac{1}{3}\right\}, \Delta_3\in \left\{\frac{r}{6}-\frac{5}{3}, \frac{r}{6}+\frac{1}{3}\right\}.$$
No matter how we combine these options, $\Delta\leq 0$.

\noindent \textbf{Subcase:} $c_2=-1$. In this subcase, $r=3q_2-1, q_2=\frac{r}{3}+\frac{1}{3}$ and $r=6q_3+2, q_3=\frac{r}{6}-\frac{1}{3}$. The minimal column dimensions are
$$\Delta_1^0=\frac{-r}{2}, \Delta_2^0=\frac{r}{3}-\frac{5}{3}, \Delta_3^0=\frac{r}{6}-\frac{7}{3}.$$
The minimal dimension is then $\Delta^0=-4$. If $\Delta_1\neq \frac{-r}{2}$, then $\Delta_1\geq 0$ and $\Delta\geq \frac{r}{2}-4$. Hence, we assume $\Delta_1=\frac{-r}{2}$. 

Let's look at changing the second column.
\begin{enumerate}
    \item $(q_2,q_2,q_2-1)$, $\Delta_2\in \left\{\frac{r}{3}-\frac{5}{3}, \frac{r}{3}+\frac{1}{3}\right\}$
    \item $(q_2,q_2-1,q_2-1,1)$, $\Delta_2\in \left\{r+1,r+3\right\}$
    \item $(q_2,q_2,q_2-2,1)$, $\Delta_2\in \left\{r-8,r-6,r-4,r-2\right\}$
\end{enumerate}
Any other partition is less than (ii) $(q_2,q_2-1,q_2-1,1)$ or (iii) $(q_2,q_2,q_2-2,1)$ and hence satisfies $\Delta_2\geq r-8$. It follows that if $\Delta_2\notin\left\{\frac{r}{3}-\frac{7}{3}, \frac{r}{3}+\frac{1}{3}\right\}$, then
$$\Delta\geq \frac{-r}{2}+(r-8)+\left(\frac{r}{6}-\frac{7}{3}\right)=\frac{2r}{3}-\frac{31}{3}.$$
Now we can assume $\Delta_2\in\left\{\frac{r}{3}-\frac{7}{3}, \frac{r}{3}+\frac{1}{3}\right\}$. 

Let's now consider changing the third column. 
\begin{enumerate}
    \item $(q_3,q_3,q_3,q_3,q_3,q_3,2)$, $\Delta_3\in \left\{\frac{r}{6}-\frac{7}{3}, \frac{r}{2}-7\right\}$
    \item $(q_3,q_3,q_3,q_3,q_3,q_3,1,1)$, $\Delta_3\in \left\{\frac{r}{6}-\frac{7}{3}, \frac{r}{6}-\frac{1}{3}, \frac{r}{2}-5,\frac{r}{2}-3\right\}$
    \item $(q_3,q_3,q_3,q_3,q_3,q_3-1,3)$, $\Delta_3\in \left\{\frac{r}{2}-7, \frac{5r}{6}-\frac{47}{3}\right\}$
\end{enumerate}
Any other partition is less than $(q_3,q_3,q_3,q_3,q_3,q_3-1,3)$ and hence satisfies $\Delta_3\geq \frac{r}{2}-7$ once $r\geq 27$. If $\Delta_3\notin \left\{\frac{r}{6}-\frac{7}{3}, \frac{r}{6}-\frac{1}{3}\right\}$, then 
$$\Delta\geq \frac{-r}{2}+\left(\frac{r}{3}-\frac{7}{3}\right)+\left(\frac{r}{2}-7\right)=\frac{r}{3}-\frac{28}{3}.$$
Now we can assume $\Delta_3\in \left\{\frac{r}{6}-\frac{7}{3}, \frac{r}{6}-\frac{1}{3}\right\}$. We have now reduced our choices of column dimensions to
$$\Delta_1=\frac{-r}{2}, \Delta_2\in\left\{\frac{r}{3}-\frac{7}{3}, \frac{r}{3}+\frac{1}{3}\right\}, \Delta_3\in \left\{\frac{r}{6}-\frac{7}{3}, \frac{r}{6}-\frac{1}{3}\right\}.$$
No matter how we combine these choices, we get $\Delta \leq 0$. 

\noindent \textbf{Subcase:} $c_2=0$. This is our last subcase and will be where we find our minimal cases (see Example \ref{ex_mincases}). In this subcase, $r=2q_1, r=3q_2$, and $r=6q_3$ with minimal column dimensions
$$\Delta_1^0=\frac{-r}{2}, \Delta_2^0=0, \Delta_3^0=\frac{-r}{6}$$
and $\Delta^0=\frac{-2r}{3}$. 

Let's look at changing the first column. 
\begin{enumerate}
    \item $(q_1,q_1)$, $\Delta_2\in \left\{\frac{-r}{2}, \frac{r}{2}\right\}$
    \item $(q_1,q_1-1,1)$, $\Delta_2\in \left\{\frac{r}{2}, \frac{3r}{2}-4\right\}$
\end{enumerate}
Any other partition is lower than (ii) $(q_1,q_1-1,1)$ and hence satisfies $\Delta_2\geq \frac{r}{2}$ once $r\geq 4$. Then unless $\Delta_1=\frac{-r}{2}$,
$$\Delta\geq \frac{r}{2}+0+\frac{-r}{6}=\frac{r}{3}.$$
Now we can assume $\Delta_1=\frac{-r}{2}$. Let's look at changing the second column.
\begin{enumerate}
    \item $(q_2,q_2,q_2)$, $\Delta_2=0$
    \item $(q_2, q_2,q_2-1,1)$, $\Delta_2\in \left\{\frac{2r}{3}-4,\frac{2r}{3}-2,\frac{2r}{3}\right\}$
    \item $(q_2,q_2-1,q_2-1,2)$, $\Delta_2\in \left\{\frac{4r}{3}-8,\frac{4r}{3}-4\right\}$
    \item $(q_2,q_2,q_2-2,2)$, $\Delta_2\in \left\{\frac{4r}{3}-12,\frac{4r}{3}-8, \frac{4r}{3}-4\right\}$
\end{enumerate}
Any other partition is less than either (iii) $(q_2,q_2-1,q_2-1,2)$ or (iv) $(q_2,q_2,q_2-2,2)$ and hence satisfies $\Delta_2\geq \frac{4r}{3}-12$. Unless $\Delta_2\in \left\{0,\frac{2r}{3}-4,\frac{2r}{3}-2,\frac{2r}{3}\right\}$,
$$\Delta\geq \frac{-r}{2}+\left(\frac{4r}{3}-12\right)+\frac{-r}{6}=\frac{2r}{3}-12.$$
Hence, we can now assume $\Delta_2\in \left\{0,\frac{2r}{3}-4,\frac{2r}{3}-2,\frac{2r}{3}\right\}$.

Let's look at changing the third column. We will find our minimal cases when $\Delta_3=\frac{5r}{6}-12$. We first look at partitions achieved from $(q_3,q_3,q_3,q_3,q_3,q_3)$ within two moves of the Young diagram.
\begin{enumerate}
    \item $(q_3,q_3,q_3,q_3,q_3,q_3)$, $\Delta_3\in \left\{\frac{r}{6}, \frac{-r}{6}\right\}$
    \item $(q_3,q_3,q_3,q_3,q_3,q_3-1,1)$, $\Delta_3\in \left\{\frac{r}{6},\frac{r}{2}-4\right\}$
    \item $(q_3,q_3,q_3,q_3,q_3-1,q_3-1,2)$, $\Delta_3\in \{\frac{r}{2}-4,\frac{r}{2}-2,\frac{5r}{6}-8,\frac{5r}{6}-10\}$
    \item $(q_3,q_3,q_3,q_3,q_3-1,q_3-1,1,1)$, $\Delta_3\in \left\{\frac{r}{2}, \frac{r}{2}-2,\frac{r}{2}-4,\frac{5r}{6}-4,\frac{5r}{6}-6,\frac{5r}{6}-8\right\}$
    \item $(q_3,q_3,q_3,q_3,q_3,q_3-2,2)$, $\Delta_3\in \left\{\frac{r}{2}, \frac{5r}{6}-12\right\}$
    \item $(q_3,q_3,q_3,q_3,q_3,q_3-2,1,1)$, $\Delta_3\in \left\{\frac{r}{2}-4,\frac{r}{2}-2,\frac{5r}{6}-10,\frac{5r}{6}-8\right\}$
\end{enumerate}

We next look at the following partitions.
\begin{enumerate}[label=(\alph*)]
    \item $(q_3,q_3,q_3,q_3,q_3-1,q_3-2,3)$, $\Delta_3\in \left\{\frac{5r}{6}-10,\frac{5r}{6}-8,\frac{7r}{6}-20,\frac{7r}{6}-18\right\}$
    \item $(q_3,q_3,q_3,q_3,q_3,q_3-3,3)$, $\Delta_3\in \left\{\frac{5r}{6}-12, \frac{7r}{6}-24\right\}$
    \item $(q_3,q_3,q_3,q_3,q_3,q_3-3,2,1)$, $\Delta_3\in \left\{\frac{5r}{6}-10,\frac{5r}{6}-8,\frac{7r}{6}-20,\frac{7r}{6}-18\right\}$
    \item $(q_3,q_3,q_3,q_3,q_3,q_3-4,4)$, $\Delta_3\in \left\{\frac{7r}{6}-24, \frac{3r}{2}-40\right\}$
\end{enumerate}
Any other partition is less than (a) $(q_3,q_3,q_3,q_3,q_3-1,q_3-2,3)$, (c) $(q_3,q_3,q_3,q_3,q_3,q_3-3,2,1)$, or (d) $(q_3,q_3,q_3,q_3,q_3,q_3-4,4)$ and hence satisfies $\Delta_3\geq \frac{5r}{6}-10$ once $r\geq 45$. Unless we use 
$$\Delta_3\in \left\{\frac{r}{6},\frac{-r}{6}, \frac{r}{2}-4,\frac{r}{2}-2, \frac{r}{2},\frac{5r}{6}-12\right\},$$
$$\Delta\geq \frac{-r}{2}+0+\left(\frac{5r}{6}-10\right)=\frac{r}{3}-10.$$
We can now assume $\Delta_3$ is one of the values above. 

We have reduced the possible column dimensions to
$$\Delta_1=\frac{-r}{2}, \Delta_2\in \left\{0,\frac{2r}{3}-4,\frac{2r}{3}-2,\frac{2r}{3}\right\}, \Delta_3\in \left\{\frac{-r}{6},\frac{r}{6}, \frac{r}{2}-4,\frac{r}{2}-2,\frac{r}{2},\frac{5r}{6}-12\right\}.$$
If $\Delta_3=\frac{-r}{6}$, then
$$\Delta\leq \frac{-r}{2}+\left(\frac{2r}{3}\right)+\frac{-r}{6}=0.$$
If $\Delta_3\neq \frac{-r}{6}$ and $\Delta_2\neq 0$, then
$$\Delta\geq \frac{-r}{2}+\left(\frac{2r}{3}-4\right)+\left(\frac{r}{6}\right)= \frac{r}{3}-4.$$
Hence we can assume $\Delta_3\neq \frac{-r}{6}$ and $\Delta_2=0$. Unless, we take $\Delta_3=\frac{5r}{6}-12$,
$$\Delta\leq \frac{-r}{2}+0+\frac{r}{2}=-2.$$
Thus we must take $\Delta_1=\frac{-r}{2}, \Delta_2=0, \Delta_3=\frac{5r}{6}-12$ with $\Delta=\frac{r}{3}-12$. We saw that this occurs when the partition for the third column is either (v) $(q_3,q_3,q_3,q_3,q_3,q_3-2,2)$ or (b) $(q_3,q_3,q_3,q_3,q_3,q_3-3,3)$. 

Now that we have run through all the cases, we have found that all configurations satisfy $\Delta\geq \frac{r}{3}-12$ with equality only possible for the examples given in Example \ref{ex_mincases}. This proves Theorem \ref{thm_quadthm_config}.

\footnotesize{\bibliography{proj.bib}}
\bibliographystyle{alpha}
\end{document}